\documentclass{amsart}

\usepackage{url}

\usepackage[centering]{geometry}

\usepackage{amsmath,amscd}
\usepackage{amsthm}
\usepackage{amssymb}
\usepackage{mathtools}
\usepackage{mathrsfs}
\usepackage{amsfonts}

\usepackage{graphicx}
\usepackage{xcolor}
\usepackage[hypcap]{caption}
\usepackage{subcaption}

\usepackage{paralist}  

\newcommand{\mc}{\mathcal}
\newcommand{\msf}{\mathsf}
\newcommand{\mbf}{\boldsymbol} 

\newcommand{\mscr}{\mathscr}

\newtheorem{thrm}{Theorem}[section] 
\newtheorem{lem}[thrm]{Lemma} 
\newtheorem{prop}[thrm]{Proposition} 
\newtheorem{cor}[thrm]{Corollary} 

\theoremstyle{definition}
\newtheorem{example}[thrm]{Example} 
\newtheorem{defn}[thrm]{Definition}

\theoremstyle{remark} 
\newtheorem{rem}[thrm]{Remark}

\numberwithin{equation}{section}

\newcommand{\C}{ \ensuremath{\mathbb{C}}}
\newcommand{\R}{ \ensuremath{\mathbb{R}}}

\newcommand{\N}{ \ensuremath{\mathbb{N}}}
\newcommand{\Z}{ \ensuremath{\mathbb{Z}}}
\newcommand{\D}{ \ensuremath{\mathbb{D}}}  

\newcommand{\al}{\alpha}
\newcommand{\eps}{\varepsilon}

\newcommand{\lam}{\lambda}
\renewcommand{\phi}{\varphi}

\newcommand{\norm}[1]{ \ensuremath{\left\lVert{#1}\right\rVert}}

\newcommand{\twonorm}[1]{ \ensuremath{\left\lVert{#1}\right\rVert}_{2}}
\newcommand{\infnorm}[1]{ \ensuremath{\left\lVert{#1}\right\rVert}_{\infty}}

\newcommand{\abs}[1]{ \ensuremath{\left\lvert{#1}\right\rvert}}
\newcommand{\inner}[2]{\ensuremath{\left\langle #1, #2 \right\rangle}}
\newcommand{\set}[1]{\left\{ #1 \right\}}
\newcommand{\setc}[2]{\left\{ #1 \mathrel{}\middle| \mathrel{} #2 \right\} }

\DeclareMathOperator{\lin}{lin} 

\newcommand{\ind}[1]{\chi_{#1}}  

\newcommand{\F}{\mathscr{F}}
\newcommand{\B}{\mathscr{B}}
\newcommand{\X}{\mscr{X}}
\newcommand{\Image}{\mathrm{Im}}  
\newcommand{\bhardy}{\ensuremath{H_{\mathscr B}^{2}} }  

\usepackage{upref}

\usepackage[pdftex,colorlinks={true},linkcolor={blue}]{hyperref}
\hypersetup{
pdftitle={Construction of Eigenfunctions for Scalar-type Operators via Laplace Averages with Connections to the Koopman Operator},
pdfauthor={Ryan Mohr, Igor Mezi\'{c}},
pdfsubject={MSC2010 Primary 47A35;  Secondary 47B33, 47B40},
pdfkeywords={Laplace Averages, Composition Operator, Koopman Operator, Spectral Measures, Off-Attractor Ergodic Theorems, Polynomials over Normed Commutative Rings}
}

\begin{document}

\title[GLA for scalar-type operators]{Construction of eigenfunctions for scalar-type operators via Laplace averages with connections to the Koopman operator}


\author{Ryan Mohr}
\address{Center for Control, Dynamical Systems, and Computation, University of California, Santa Barbara\\ Santa Barbara, CA, USA}
\email{mohrrm@engr.ucsb.edu}
\thanks{The authors gratefully acknowledge support from both the Army Research Office (ARO) through grant W911NF-11-1-0511 under the direction of program manager Dr.\ Sam Stanton and the Air Force Office of Scientific Research (AFOSR) through grant FA9550-09-1-0141 under the direction of program manager Dr.\ Bob Bonneau.}

\author{Igor Mezi\'c}
\address{Center for Control, Dynamical Systems, and Computation \& Department of Mathematics, University of California, Santa Barbara\\ Santa Barbara, CA, USA}
\email{mezic@engr.ucsb.edu}

\subjclass[2010]{Primary 47A35; Secondary 47B33, 47B40}

\keywords{Laplace averages, composition operator, Koopman operator, spectral measures, off-attractor ergodic theorems, polynomials over normed commutative rings, dynamical systems}

\date{}

\dedicatory{}

\commby{}

\begin{abstract}
This paper extends Yosida's mean ergodic theorem in order to compute projections onto non-unitary eigenspaces for spectral operators of scalar-type on locally convex linear topological spaces.  For spectral operators with dominating point spectrum, the projections take the form of Laplace averages, which are a generalization of the Fourier averages used  when the spectrum is unitary.  Inverse iteration and Laplace averages project onto eigenspaces of spectral operators with minimal point spectrum.  Two classes of dynamical systems --- attracting fixed points in $\mathbb{C}^{d}$ and attracting limit cycles in $\mathbb{R}^{2}$ --- and their respective spaces of observables are given for which the associated composition operator is spectral.  It is shown that the natural spaces of observables are completions with an $\ell^{2}$ polynomial norm of a space of polynomials over a normed unital commutative ring.  These spaces are generalizations of the Hardy spaces $H^{2}(\mathbb{D})$ and $H^{2}(\mathbb{D}^{d})$.  Elements of the ring are observables defined on the attractor --- the fixed point or the limit cycle, in our examples.  Furthermore, we are able to provide a (semi)global spectral theorem for the composition operators associated with a large class of dissipative nonlinear dynamical systems; any sufficiently smooth dynamical system topologically conjugate to either of the two cases above admits an observable space on which the associated Koopman operator is spectral.  It is conjectured that this is generically true for systems where the basin of attraction can be properly ``coordinatized''.
\end{abstract}

\maketitle



\section{Introduction}\label{sec:intro}

Consider a discrete time, dynamical system $(\X,\Phi)$, where $\X$ is some finite dimensional normed space and $\Phi : \X \to \X$ is the generator of the flow map.  When $\Phi$ is linear and (block) diagonalizable, everything is known about the behavior of the system due the spectral nature of the operator.  Unfortunately, such a global decomposition is not generally available in the nonlinear case and the analysis proceeds with the geometric methods introduced by Poincar\'{e}.  Spectral methods that are used in this framework are typically local in nature and revolve around perturbing about particular solutions and looking at the spectrum of the linear part of the perturbed dynamics.  

However, by changing our viewpoint from the geometric picture to a functional/analytic one, a global spectral theory can often be pursued for the nonlinear case.  This is accomplished by considering composition operators acting on a space of observables $\F$ (scalar functions having domain $\X$).  To the map $x\mapsto \Phi(x)$, we associate the (formally defined) operator $U_{\Phi}$
	\begin{equation}\label{eq:composition-operator}
	f \mapsto U_{\Phi}f := f \circ \Phi .
	\end{equation}
The operator $U_{\Phi}$ is often called the Koopman operator associated with $\Phi$ due to B.O. Koopman's work with operators of this form on Hamiltonian systems \cite{Koopman:1931ug}.  We will use ``Koopman operator'' and ``composition operator'' interchangeably.  $\Phi$ is called the symbol of the operator.  Questions of what $\F$ should be and how the properties of $U_{\Phi}$ relate to properties of the dynamical system are of the most interest.  If $\F$ has a linear structure and is closed under the operation \eqref{eq:composition-operator}, the composition operator is a linear operator, albeit infinite dimensional in general.  As such, a spectral theory can be pursued.  For example, when $\Phi$ is an automorphism on the compact metric space $\X$ preserving a finite measure $\mu$ and $\F$ is taken as the Hilbert space of observable $L^{2}(\X,\mu)$, $U_{\Phi}$ is a unitary operator on a Hilbert pace and thus has a spectral decomposition \cite{Bachman:2000ul}.  For a more general system though, the structure of $\F$ and whether the operator is spectral (either, completely ``diagonalizable'' or in the sense of Dunford) is a more difficult question to answer.

Spectral properties of composition operators have proven useful in analyzing concrete systems.  By constructing eigenfunctions, either explicitly or implicitly, the practitioner has gained insight into the behavior of their particular system.  Typically, rather than a full spectral decomposition, a so-called Koopman mode analysis is pursued \cite{Mezic:2005ug}.  This amounts to considering the spectral decomposition from a specific initial condition.  This is done in the fluids literature via the DMD algorithm \cite{Schmid:2008wv,Schmid:2010ba} (whose connection to the Koopman operator was given in \cite{Rowley:2009ez}) and its utility has been born out in a number of studies; e.g. see \cite{Rowley:2009ez,Schmid:2011ve,Seena:2011ft} or the recent review article \cite{Mezic:2012tx}, and the references therein, for a subset.  Other applied studies, such as energy efficiency in buildings \cite{Georgescu:2012tt,Eisenhower:2010tv}, power-systems analysis \cite{Susuki:2012gk,Susuki:2011ef}, and neurodynamics \cite{Mauroy:2012by} have used spectral properties and the Koopman mode decomposition of the associated composition operator.

In most of these studies it is implicitly assumed that the dynamics considered are on the attractor and, therefore, the eigenfunctions corresponding to unitary (unit modulus) eigenvalues are important.  A general procedure for constructing such eigenfunctions is through the use of infinite Fourier averages
	\begin{equation}\label{eq:fourier-average}
	\lim_{n \to \infty} n^{-1} \sum_{k=0}^{n-1} e^{-i\omega k} U_{\Phi}^{k}, \quad (\omega \in \R).
	\end{equation}
The existence of these limits can be guaranteed in certain situations by various ergodic theorems, Wiener and Wintner's extension of Birkhoff's pointwise ergodic theorem \cite{Wiener:1941wy,Assani:2003us} and Yosida's extension of the mean ergodic theorems \cite{Yosida:1978ul} having primacy.

Often though, it is quite useful to know the spectrum and eigenfunctions for non-measure preserving dynamics.  For example, the recent work in neurodynamics on isochrons and isostables \cite{Mauroy:2013jb} leverage the non-unitary spectrum of the Koopman operator.  The works of Cowen, MacCluer, and others on the spectra of composition operators with analytic symbols acting on (weighted) Hardy Hilbert spaces $H^{2}(\D)$ or Bergman spaces $A^{p}(\D)$ in the complex unit disc (or more generally, the polydisc) can be considered as other examples \cite{Cowen:1994cw,Singh:1993vt,Ridge:1973ud,Galindo:2009kc,Aron:2004kr,Hyvarinen:2013hb}.  Unfortunately, it is not immediately clear that such dynamical systems arise naturally when considering applied problems and, additionally, there seems not to be a clear general procedure on constructing non-unitary eigenfunctions as there is for unitary eigenfunctions through the application of \eqref{eq:fourier-average}.  

This paper and its companion aim to partially fill this gap.  In pursuing this goal, we first prove results in an abstract setting of spectral operators of scalar type acting on locally convex linear topological spaces\footnote{In the companion paper, we consider spectral operators of finite type (in the sense of Dunford) and more general forms of averages.} and then specialize to composition operators associated to particular dynamical systems.  In this paper, our main results include a general method for constructing eigenfunctions for scalar-type spectral operators having non-unitary eigenvalues and, additionally, a method for constructing spaces of observables corresponding to a dynamical system for which the associated Koopman operator is spectral.  Analogous to the viewpoint that the Laplace transform is a generalization of the Fourier transform, Laplace averages (to be defined later) will take the place of the Fourier averages of \eqref{eq:fourier-average} in the construction of the projection operators.  We do mention that the generalized Laplace averages of this paper also appear in \cite{Mauroy:2013jb}; however, in that work, a rigorous analysis of when and on what spaces these averages exist is missing.  In the course of proving our results, it will become clear that there must be restrictions on the placement of the point spectrum of the operator and the observables to which the Laplace averages are applied.  Later on, it will be shown that such restrictions are naturally satisfied for certain composition operators and their associated dynamical systems.  We do this by explicitly constructing the space of observables for two basic types of attractors important in applied problems --- fixed points and limit cycles --- and showing the associated composition operator is spectral in a function space defined on bounded, positively invariant, simply connected subsets in the basin of attraction.  The eigenfunctions that naturally arise for these systems live in the space of polynomials over normed unital commutative rings; that is, polynomials where the coefficients of the indeterminates are elements from a commutative ring\footnote{This is not to be confused with the space of Banach algebra polynomials which is formed from the space of polynomials over $\C$ whose indeterminates are replaced by elements from a Banach algebra.}.  We can complete these spaces under an $\ell^{2}$ polynomial norm.  We connect these completions back to Hardy spaces in the disc, which have been the focus of other work on composition operators, for example the work of Cowen and MacCluer \cite{Cowen:1994cw}.  In particular, the spaces of observables we construct are generalizations of the Hilbert Hardy spaces in the unit disc or polydisc.  These comments are contained in remarks \ref{rem:hardy-space} and \ref{rem:connection-with-work-of-cowen}.

This paper is structured as follows.  Section \ref{sec:preliminaries} states and discusses Yosida's mean ergodic theorem in order to motivate the assumptions on the spectrum that will be required to prove our results.  The section also defines and gives basic properties of spectral measures in locally convex linear topological spaces which are the second main element in our analysis.

Section \ref{sec:gla-lclts} contains our main results on constructing eigenfunctions for non-unitary eigenvalues of spectral operators.  For operators with dominating point spectrum (definition \ref{defn:dominating-point-spectrum}), the Generalized Laplace analysis (GLA) theorem (theorem \ref{eq:gla-lcs-forward}) is the main result.  It gives an iterative procedure to compute projection onto eigenspaces.  In essence, one needs to project onto eigenfunctions corresponding to the peripheral point spectrum first, subtract these out from the dynamics, and then repeat on the reduced system.  Conceptually, this is similar to applying a perturbation to the system that moves all the peripheral eigenvalues to (or close to) 0.  However, for many systems of interest, the point spectrum is not dominating.  The stable directions correspond to eigenvalues inside the unit circle, whereas the spectrum corresponding to the attractor is contained in the unit circle and may posses continuous parts.  In order to leverage the already proven GLA results, we use inverse iteration.  This transforms the spectrum into the previous case of dominating point spectrum.  This result is contained in theorem \ref{prop:gla-inverse-iteration}.

Section \ref{sec:dyn-sys} connects the abstract results on general spectral operators with dynamical systems and the Koopman operator.  Polynomials over normed commutative rings are defined first since they arise naturally from eigenfunctions of the Koopman operator.  It is also shown that under a completion with an $\ell^{2}$ polynomial norm that the completed space is locally sequentially weakly compact as is required to apply the GLA theorems.  This holds true as long as the Banach space in which the ring is dense is separable and reflexive.  This result is contained in proposition \ref{prop:lswc}.  Section \ref{sec:examples} constructs eigenfunctions and spaces of observables for attracting hyperbolic fixed points in $\C^{d}$ and limit cycles in $\R^{2}$.  Topological conjugacies are leveraged to use eigenfunctions corresponding to the linearized dynamics to construct eigenfunctions and spaces of observables for the nonlinear system.  It is shown that these spaces are identifiable with spaces of polynomials over normed unital commutative rings.  Spectral measures for the systems are constructed that satisfy the properties assumed in section \ref{sec:preliminaries}.  Additionally, a short remark connects the space of observables of an attracting fixed point for the $d=1$ case with Hardy Hilbert spaces in the unit disc $H^{2}(\D)$ that are often studied in the context of composition operators with analytic symbols defined in the unit disc.  Loosely, the spaces we construct can be viewed as generalizations of these spaces.

Some concluding remarks are given in section \ref{sec:conclusions} and some future directions are pointed out.

\section{Preliminaries}\label{sec:preliminaries}

\subsection{Yosida's Mean Ergodic theorem.}

Let $\F$ be a complex vector space.  Let $\Gamma$ be an index set and $\set{p_{\gamma} : \F \to \R}_{\gamma \in \Gamma}$ a separating family of seminorms.  Recall, that together these mean that for all $\gamma \in \Gamma$, scalars $\al$, and $f,g \in \F$ that
\begin{inparaenum}[(i)]
\item $p_{\gamma}(f) \geq 0$, \label{item:nonnegative}
\item $p_{\gamma}(\al f) = \abs{\al} p_{\gamma}(f)$, \label{item:homogeneity}
\item $p_{\gamma}(f + g) \leq p_{\gamma}(f) + p_{\gamma}(g)$, and \label{item:subadditivity}
\item for every $f \neq 0$, there is a $p_{\gamma(f)} \in \Gamma$ such that $p_{\gamma(f)}(f) \neq 0$.
\end{inparaenum}

Define on $\F$ the locally convex topology (LCT) generated by the seminorms.  Neighborhood bases of the LCT have the form $U(f; \eps) := \setc{ g \in \F }{ p_{\gamma_{i}}(g - f) < \eps, \forall i \in I }$, where $I \subset \Gamma$ is an arbitrary finite subset and $\eps$ is an arbitrary positive number.  Note that a sequence $\set{f_{n}} \subset \F$ converges to $f \in \F$ if and only if for all $\gamma \in \Gamma$, $p_{\gamma}(f_{n} - f) \to 0$ as $n\to \infty$.  This topology will be called the strong topology on $\F$ and with this topology $\F$ is a locally convex linear topological space (LCLTS).

Denote by $\mscr L(\F)$ the set of all linear operators on $\F$ continuous with respect to the strong topology.  The set $\F^{*}$, the topological dual to $\F$, is the space of all linear functionals continuous with respect to the locally convex topology.  The weak topology on $\F$, denote by $\sigma(\F, \F^{*})$ is generated by the family of seminorms $\setc{ p_{\phi} = \abs{\phi} }{ \phi \in \F^{*} }$.

Fix $U \in \mscr L(\F)$.  The resolvent $\rho(U) \subset \C$ of $U$ is the set of $\xi \in \C$ such that $\Image(\xi I - U)$ is dense in $\F$ and $(\xi I - U)^{-1}$ exists and is continuous.  The complement of this set is the spectrum, denoted $\sigma(U)$.  The point spectrum $\sigma_{p}(U)$ is the set of $\lam \in \sigma(U)$ for which $(\lam I - U)$ does not have an inverse. When the operator $U$ is understood, it will be dropped from the notation.  We denote for $\Lambda \subset \sigma$
	\begin{equation}\label{eq:spectral-modulus}
	\abs{\Lambda} := \setc{ \abs{\lam} }{\lam \in \Lambda}
	\end{equation}
and the spectral radius of $\Lambda$ as
	\begin{equation}\label{eq:spectral-radius}
	\infnorm{\Lambda} := \sup_{\lam \in \Lambda} \abs{\lam}.
	\end{equation}

\begin{defn}[Equicontinuous family of linear operators]
Consider a fixed linear operator $U\in \mscr L(\F)$.  The family $\set{ U^{n} }_{n \in \N}$ is called equicontinuous if for every continuous (with respect to the LCT) seminorm $p:\F \to \R$, there is a continuous seminorm $p':\F \to \R$ such that
	\begin{equation}\label{eq:equicontinuity}
	\sup_{n \geq 1} p(U^{n} f) \leq p'(f).
	\end{equation}
for every $f \in \F$.
\end{defn}

\begin{defn}[Averaging operators]
For each $n \in \N$, define the averaging operator $A_{n} : \mscr L(\F) \to \mscr L(\F)$ as 
	\begin{equation}
	A_{n}(U) := n^{-1} \sum_{k=0}^{n-1} U^{k}.
	\end{equation}
\end{defn}


One of the major tools of this paper is Yosida's extension of the mean ergodic theorem.

\begin{thrm}[Yosida's mean ergodic theorem, \cite{Yosida:1978ul}]
Let $\set{U^{n}}_{n\in\N}$ be an equicontinuous family of linear operators defined on a locally convex linear topological space $\F$.  Fix $\psi \in \F$.  If there exists as subsequence $\set{ n_{i} } \subset \N$ and an $\psi_{0} \in \F$ such that $\displaystyle \lim_{i\to\infty} \inner{ A_{n_{i}}(U) \psi - \psi_{0}}{ \phi^{*} } = 0$ for all $\phi^{*} \in \F^{*}$, then $U\psi_{0} = \psi_{0}$ and $\displaystyle \lim_{n\to \infty} A_{n}(U) \psi = \psi_{0}$ exists in the strong topology.

Futhermore, if $\F$ is locally sequentially weakly compact, then $\displaystyle P := \lim_{n\to \infty} A_{n}(U)$, with the limit taken in the strong operator topology, defines a continuous projection operator onto $N(I - U)$ commuting with $U$ and giving the direct sum decomposition $\F = N(I - U) \oplus \overline{\Image(I - U)}$.
\end{thrm}

\subsubsection{Remarks on Yosida's mean ergodic theorem.} 

\begin{compactenum}[(i)]
\item The theorem states that if the averages of an element have a weakly convergent subsequence, this implies that the full average converges strongly to the limit.

\item The mean ergodic theorem applies, in particular, to when $\F$ is a Banach space as this is just a specialized example of a LCLTS. 

\item Note that the first result does not define a projection operator on $\F$ since some $\psi$ may not have a weakly convergent subsequence of averages.  To define the spectral projection and get a direct sum decomposition of the space, every $\psi \in \F$ must have some subsequence of $\set{A_{n}(U) \psi}$ that is weakly convergent.  This is guaranteed by $\F$ being locally sequentially weakly compact.  In particular, if $\F$ is a separable reflexive Banach space, it is locally sequentially weakly compact.  This follows directly from the sequential version of the Banach-Alaoglu theorem.

\item If $\abs{ \lam } = 1$ and $\set{ U^{n} }$ is equicontinuous, then $\set{ \lam^{-n}U^{n} }_{n\in\N}$ is also equicontinuous.  If the sequence $\{ A_{n}(\lam^{-1}U) \psi \}_{n\in\N}$ has a weakly convergent subsequence, then 
	\begin{equation*}
	\lim_{n\to \infty}  A_{n}(\lam^{-1}U) \psi  = \psi_{\lam} \in \F
	\end{equation*}
exists, $\lam^{-1} U \psi_{\lam} = \psi_{\lam}$, and hence $\psi_{\lam}$ is an element of the eigenspace at $\lam$.

\item\label{item:no-limits-from-below} Suppose $\F$ is a Banach space and suppose $\set{ U^{n} }_{n\in\N}$ ($U \in \mscr L(\F)$) is not equicontinuous in the norm (e.g. $1 < \norm{U} < \infty$).  We always have the spectral radius as a lower bound on the norm of $U$:
	\begin{equation*}
	\infnorm{ \sigma(U) } := \sup_{\lam \in \sigma(U) } \abs{\lam} \leq \norm{U}.
	\end{equation*}
To get an element of an eigenspace using the above results, there must be some $\mu \in \sigma(U)$ such that $ \set{ \left(\mu^{-1}U \right)^{n}}$ is equicontinuous.  This happens only if $\abs{\mu} = \norm{U}$.

In particular, suppose $\sigma(U)$ has a sequence of eigenvalues $\set{ \lam_{j} }$ such that $\set{ \abs{ \lam_{j} } }$ is a strictly increasing sequence, $\abs{\lam_{j}} < \norm{U}$, and $\displaystyle \lim_{j\to\infty} \abs{\lam_{j}} = \norm{U}$.  For $\lam_{j}$, let $\phi \in N(\lam_{j+1} I - U)$ and $\norm{\phi} = 1$.  Then
	\begin{equation*}
	\norm{ \left( \lam_{j}^{-1} U \right)^{n} } \geq \norm{ \left( \lam_{j}^{-1} U \right)^{n} \phi } = \abs{{\lam_{j+1}} /{\lam_{j}}}^{n}
	\end{equation*}
Therefore $\set{\left( \lam_{j}^{-1} U \right)^{n}}$ cannot be an equicontinuous family.

\item\label{item:no-large-continuous-spectrum} Suppose $\norm{U} \leq 1$ and $\F$ is a locally sequentially weakly compact Banach space.  Suppose $\lam_{1,2}$ are in the point spectrum $\sigma_{p}(U)$ and satisfy $\abs{\lam_{2}} < \abs{\lam_{1}} = \norm{U}$.  Also assume that $U$ has a continuous portion of the spectrum, $\mu \in \sigma_{c}(U)$, and $\abs{\lam_{2}} < \abs{\mu} < \abs{\lam_{1}} = \norm{U}$.  The above results can be applied to any $f \in \F$ using $\lam_{1}^{-1}U$ to get a projection onto $N(\lam_{1} - U)$.  

We would also like to use the same procedure to get the projections onto $N(\lam_{2} I - U)$.  Recall that we had the direct sum decomposition $\F = \overline{\Image(\lam_{1} I - U)} \oplus N(\lam_{1} I - U)$.  Consider elements in $\overline{\Image(\lam_{1} I - U)}$ and restrict $U$ to this subspace (which we will denote as $U_{2}$).  We cannot use the averaging above to get the projection since the continuous part of the spectrum prevents $\set{(\lam_{2}^{-1} U_{2})^{n}}$ from being equicontinuous.  This is because the norm of $U_{2}$ satisfies $\abs{\lam_{2}} < \abs{\mu} \leq \norm{U_{2}}$.  Therefore, if we want to apply the averaging procedure to compute the projections, $U$ must either be restricted to subspaces not including elements corresponding to non-point spectrum or all elements of the point spectrum must have greater modulus than other parts of the spectrum.  Additionally, if $\abs{ \sigma_{p}(U) }$ is a discrete set, then we can guarantee that some eigenvalue's modulus attains the spectral radius.  This condition means that the point spectrum is contained in separated circles in the complex plane.
\end{compactenum}

Given these remarks, in order to apply the mean ergodic theorem in an iterative manner, some assumptions on the spectrum must be made.  Loosely, the point spectrum must be larger than the non-point spectrum and additionally, whenever any set of eigenvalues is removed from the spectrum, there is an eigenvalue that achieves the supremum over the remaining set (if the remaining set contains an eigenvalue).  To this end we define the following objects.  The first definition is standard.

\begin{defn}[Peripheral spectrum]
$\lam \in \sigma(U)$ is in the peripheral spectrum of $U$ if the modulus of $\lam$ is equal to the spectral radius;  $\displaystyle \abs{\lam} = \infnorm{ \sigma(U) }$.  If $B \subset \C$, then the peripheral spectrum of $\sigma(U) \cap B$ is defined as $\setc{ \lam \in \sigma(U) \cap B }{ \abs{\lam} = \infnorm{\sigma(U) \cap B} }$.
\end{defn}

\begin{defn}[Dominating point spectrum]\label{defn:dominating-point-spectrum}
For $r > 0$, let $\D_{r}$ be the open disc of radius $r$ centered at 0 in the complex plane and let $\sigma(U; \D_{r}) := \D_{r} \cap \sigma(U)$.  If there exists an $R > 0$ such that $\sigma(U) \setminus \D_{R}$ is not empty and for every $r > R$ we have:
\begin{compactenum}[(i)]
\item if $\sigma(U; \D_{r}) \cap \sigma_{p}(U) \neq \emptyset$, then the peripheral spectrum of $\sigma(U;\D_{r})$ is not empty, and
\item the set $\sigma(U) \setminus \D_{r}$ consists only of eigenvalues, \label{item:peripheral-eigenvalues}
\end{compactenum}
then $U$ is said to have dominating point spectrum.
\end{defn}

We note that the second condition of this definition implies that the peripheral spectrum of $\sigma(U; \D_{r})$ consists only of eigenvalues if it is not empty.  Basic properties of a dominating point spectrum are given in the following lemma.

\begin{lem}\label{lem:dominating-spectrum-properties}
Assume $U\in \mscr L(\F)$ has a dominating point spectrum.  Then 
\begin{compactenum}[(i)] 
\item every pair $\lam \in \sigma_{p}$ and $\xi \in \sigma \setminus \sigma_{p}$ with $\abs{\lam} > R$ satisfies $\abs{\xi} \leq \abs{\lam}$,
\item\label{item:no-increasing-limit} no sequence of eigenvalues $\set{\lam_{n}}_{n\in\N} \subset \sigma_{p}$ has both strictly increasing moduli and satisfies ``$\displaystyle \lim_{n \to \infty} \abs{\lam_{n}}$ exists'', and
\item\label{item:dominating-subset} if $\Lambda\subset \sigma_{p}$ is any set of eigenvalues, then $\sigma \setminus \Lambda$ has dominating point spectrum.
\end{compactenum}
\end{lem}

\begin{proof}
\begin{compactenum}[(i)]
\item Either $\abs{\xi} \leq \abs{\lam}$ or $\abs{\xi} > \abs{\lam} > R$. In the first case, this is what we are trying to prove, so assume the second case holds.  Then $\xi \in \sigma \setminus \D_{\abs{\lam}}$ which contradicts that this set contains only eigenvalues.  This contradiction gives the result.

\item  Consider the sequence $\set{ \lam_{n} }_{n\in \N} \subset \sigma_{p}$ satisfying $\abs{\lam_{n}} < \abs{\lam_{n+1}}$ for all $n\in \N$ and $\displaystyle \lim_{n\to \infty} \abs{\lam_{n}} = r$.  Then $\infnorm{ \sigma(U;\D_{r} ) } = r$ and no eigenvalue having modulus equal to $r$ is in $\sigma(U;\D_{r} )$ since $\D_{r}$ is open.  Therefore, the peripheral spectrum of $\sigma(U;\D_{r})$ contains no eigenvalues.  But, $\sigma(U;\D_{r}) \cap \sigma_{p}$ is not empty.  This contradicts the first condition in definition \ref{defn:dominating-point-spectrum}.

\item Assume that $\sigma_{p} \setminus \Lambda$ is not empty, since otherwise the conditions for a dominating point spectrum are trivially fulfilled.  Fix $r \in \R^{+}$.  If $(\sigma \setminus \Lambda) \cap \D_{r}$, contains no eigenvalues, then we need to do nothing.  Suppose, then, that this set does contain an eigenvalue.  Put $\al = \sup \abs{\lam}$, where the supremum is taken over all eigenvalues in $(\sigma \setminus \Lambda) \cap \D_{r}$, and let $\set{\lam_{n}}_{n\in\N}$ be an increasing, maximizing sequence of eigenvalues in this set; i.e. 
\begin{inparaenum}[(a)]
\item $\forall n\in \N$, $\abs{\lam_{n}} \leq \abs{\lam_{n+1}}$ and
\item $\displaystyle \lim_{n\to\infty} \abs{\lam_{n}} = \al$.
\end{inparaenum}
Since the sequence of moduli of eigenvalues has a limit, then by result (\ref{item:no-increasing-limit}) of this lemma, this sequence cannot be strictly increasing; there must some point $n\in \N$ such that $\abs{\lam_{n}} = \abs{\lam_{n+1}}$.  We claim that there is an $M \in \N$ such that $\abs{\lam_{M}} = \abs{\lam_{M+m}}$ for all $m \geq 1$.  Assume to the contrary.  Put $N_{1} = 1$.  Then there exists $m_{1} \geq 1$ such that $\abs{\lam_{N_{1}}} < \abs{ \lam_{N_{1}+m_{1}} }$.  Put $N_{2} = N_{1} + m_{1}$.  Suppose, we have chosen $N_{1}, \dots, N_{j}$ in this manner.  Then there exists $m_{j} \geq 1$ such that $\abs{\lam_{N_{j}} } < \abs{\lam_{N_{j} + m_{j}} }$.  Put $N_{j+1} = N_{j} + m_{j}$.  By induction, we have constructed a subsequence $\set{ \lam_{N_{j}} }_{j \in \N}$ that has strictly increasing modulus and $\displaystyle \lim_{j\to \infty} \abs{ \lam_{N_{j}}} = \al$.  This contradicts (\ref{item:no-increasing-limit}).  Therefore,  some tail of $\set{ \lam_{n} }$ has constant modulus (namely $\al$).  Therefore, a peripheral eigenvalue exists.  Finally, $(\sigma \setminus \Lambda) \setminus \D_{r}$ consists only of eigenvalues since $(\sigma \setminus \Lambda) \setminus \D_{r} \subset \sigma \setminus \D_{r}$ and $\sigma \setminus \D_{r}$ only contains eigenvalues.
\end{compactenum}
\end{proof}

In particular, assuming that $U$ has a dominating point spectrum guarantees that situation in remark \eqref{item:no-large-continuous-spectrum} on the mean ergodic theorem does not occur.  Result \eqref{item:no-increasing-limit} of lemma \ref{lem:dominating-spectrum-properties} rules out the case in remark \eqref{item:no-limits-from-below} on the mean ergodic theorem.

\subsection{Spectral measures in locally convex linear topological spaces.}
As this paper is concerned with general procedures in constructing eigenfunctions for spectral operators of scalar type on LCLT spaces, we need the concept of a spectral measure for these spaces.  This machinery is contained in this section.  This information can be found in \cite{Walsh:1965vn} and \cite{Schaefer:1962dl}.

Let $(\F,\tau)$ be a locally convex linear topological space with topology $\tau$, $\mscr S$ is a set, and $\Sigma$ a $\sigma$-algebra of subsets of $\mscr S$.

\begin{defn}[Spectral measure triple \cite{Walsh:1965vn}]\label{defn:spectral-measure}
A triple $(\mscr S, \Sigma, \mu)$, where $\mu$ is a set function from $\Sigma$ to $\mscr L(\F)$ which is countable additive in the weak operator topology, that satisfies
\begin{compactenum}[(i)]
\item $\mu(\mscr S) = I \in \mscr L(\F)$ 
\item for $A_{1}$ and $A_{2}$ in $\Sigma$, $\mu(A_{1} \cap A_{2}) = \mu(A_{1}) \cdot \mu(A_{2}) \in \mscr L(\F)$
\end{compactenum}
is called a spectral measure triple.
\end{defn}

The spectral measure is said to be equicontinuous if the values in $\mscr L(\F)$ that $\mu$ takes on $\Sigma$ are equicontinuous.  If $(\mscr S, \Sigma, \mu)$ is an equicontinuous spectral measure triple and $\F$ is sequentially complete in its topology, the integrals of $\C$-valued, bounded, $\Sigma$-measurable functions can be defined \cite{Walsh:1965vn}.  The integral is multiplicative since for simple $f$ and $g$, $\int fg d\mu = \int f d\mu \cdot \int g d\mu$ follows from condition (2) in the definition of the spectral measure triple.  Hence the map $f \mapsto U_{f}$ from (equivalence classes) of $\C$-valued, bounded, $\Sigma$-measurable functions to $\mscr L(\F)$ is a homomorphism.  The integral gives a representation of the bounded $\C$-valued $\Sigma$-measurable functions as an algebra of linear operators on $\F$.  This algebra of operators is called the \emph{spectral algebra associated with $(\mscr S, \Sigma, \mu)$} and will be denoted by $\mscr A$.  Elements of this algebra are spectral operators and will be called spectral elements.

By a change of measure $E = \mu \circ f^{-1}$, we get the familiar representation of a scalar-type spectral operator (in the sense of Dunford) as an integral against a spectral measure with domain in $\mc B(\C)$;
	\begin{equation}\label{spectral-operator}
	U = \int_{\mscr S}  f(s) d\mu(s) = \int_{\C} z E(dz) .
	\end{equation}
The measure $E$ depends only on $U$ and its support is the spectrum of $U$.  The support of $E$ is contained in a compact set since $f$ is bounded by assumption.

In order to apply the mean ergodic theorem, we will need to guarantee that certain families of spectral operators are equicontinuous.  The most useful result in this regard is the following proposition, due to Walsh.

\begin{prop}[Proposition 2.1, Walsh \cite{Walsh:1965vn}]\label{prop:equicontinuous-family}
Let $\F$ be a locally convex space and $(\mscr S, \Sigma, \mu)$ an equicontinuous spectral measure triple.  Then the set
	\begin{equation}\label{eq:equicontinuous-family}
	\setc{ \int_{\mscr S} f d\mu }{ \text{$f$  $\Sigma$-measurable}, 0 \leq \abs{f} \leq 1 }
	\end{equation}
is an equicontinuous family.
\end{prop}

%

Schaefer \cite{Schaefer:1962dl} has given an operational calculus for spectral measures.  If $\mscr A \subset \mscr L(\F)$ is a weakly semi-complete locally convex algebra, $U = \int f d\mu$ a spectral element and $g : \C \to \C$ a bounded complex Baire function, the integral $\int_{\mscr S} (g \circ f) d\mu$ defines another spectral element
	\begin{equation}\label{eq:operational-calculus-lct}
	g(U) = \int_{\mscr S} (g \circ f) d\mu = \int_{\C} g(z) E(dz) \in \mscr A .
	\end{equation}
This is the operational calculus for spectral elements in locally convex spaces.  If $g$ is continuous on the spectrum of $U$, we get the spectral mapping theorem $\sigma[ g(U) ] = g[ \sigma(U) ]$ \cite{Schaefer:1962dl}.    The spectral algebras we will be most concerned about are algebras of the form
	\begin{equation*}
	\mscr A(f) := \setc{ \int_{\mscr S} (g \circ f) d\mu }{ g: \C \to \C \text{ is a polynomial}}.
	\end{equation*}
Each polynomial is bounded since $\Image(f)$ is bounded.

The reader may wonder what is gained from defining the abstract spectral measure triple $(\mscr S, \Sigma, \mu)$ in the case of locally convex linear topological spaces, rather than the form $\int_{\sigma(U)} z E(dz)$ that is familiar from scalar-type spectral operators on Banach spaces \cite{Dunford:1954we}.   This abstract framework allows, for example, the treatment of products of commuting operators with random spectrum.

\begin{example}
Let $\mscr S = \set{ 0, 1 }$.  Let $\set{ f_{1}, f_{2}}$ be a pair of functions from $\mscr S$ into $\C$.  The spectral measure $\mu$ takes values as a projection operator on $\F$.  Assume that both $P_{0} = \mu(\set{0})$ and $P_{1} = \mu(\set{1})$ are  projections onto one-dimensional subspaces and $\F = {\Image(P_{0})} \oplus {\Image(P_{1})}$.

Let $f_{1}(0) = \lam_{0}$ and $f_{1}(1) = \lam_{1}$ and similarly for $f_{2}$ but with $\omega$ replacing $\lam$.  To this pair of functions is associated a pair of operators $U_{1}$ and $U_{2}$ acting on $\F$.  Both of these operators have eigenspaces corresponding to ${\Image(P_{0})}$ and ${\Image(P_{1})}$.  Consider a finite length sequence consisting of $f_{1}$ and $f_{2}$ with $f_{1}$ appearing $m_{1}$ times and $f_{2}$ appearing $m_{2}$ times.  By the multiplicative property of the spectral integral and the commutative property of the functions
	\begin{equation*}
	(U_{i_{1}}\cdots U_{i_{m_{1}+m_{2}}} )
	=  \int_{\mscr S} (f_{i_{1}} \cdots f_{i_{m_{1}+m_{2}}})(s) d\mu(s) 
	=  \int_{\mscr S} f_{1}^{m_{1}}(s) \cdot f_{2}^{m_{2}}(s) d\mu(s) 
	= \sum_{s=0}^{1} (\lam_{s}^{m_{1}}\omega_{s}^{m_{2}} ) P_{s}.
	\end{equation*}
where $i_{j} \in \set{ 1,2 }$ for $j=1,\dots, m_{1}+m_{2}$.  The operator resulting from the product is a spectral operator with eigenvalues $\gamma_{1} = \lam_{1}^{m_{1}}\omega_{1}^{m_{2}}$ and $\gamma_{2} = \lam_{2}^{m_{1}}\omega_{2}^{m_{2}}$.  It is straightforward to take a larger collection of functions so that we can get more possible composite eigenvalues.
\end{example}

%
%
%

Projections onto different parts of the spectrum take the form of integrating against indicator functions.

\begin{lem}\label{eq:spectral-projections-lct}
Let $\Gamma$ be a Baire subset of $\C$.  Let $\ind{\Gamma} : \sigma(U) \to \C$ be the indicator function for $\Gamma \cap \sigma(U)$.  Then $P_{\Gamma} = \int_{\sigma(U)} \ind{\Gamma}(z) E(dz) = \int_{\Gamma \cap \sigma(U)} E(dz)$ is a projection operator.
\end{lem}

\begin{proof}
Since $\ind{\Gamma} \circ f$ is an indicator function for $f^{-1}(\Gamma)$, by the multiplicative property of the integral
	\begin{equation*}
	P^{2}_{\Gamma} = \left[\int_{\mscr S} (\ind{\Gamma} \circ f) d\mu \right] \cdot \left[\int_{\mscr S} (\ind{\Gamma} \circ f) d\mu\right] = \int_{\mscr S} (\ind{\Gamma} \circ f)^{2} d\mu = \int_{\mscr S} (\ind{\Gamma} \circ f) d\mu = P_{\Gamma}.
	\end{equation*}
\end{proof}

\section{Generalized Laplace Analysis in LCLT Spaces}\label{sec:gla-lclts}

Let $(\mscr S, \Sigma, \mu)$ be a spectral measure triple, $\F$ a separable locally sequentially weakly compact locally convex linear topological space.  Fix a bounded $\C$-valued $\Sigma$-measurable function $f$ and denote the associated spectral operator by $U$
	\begin{equation}
	U = \int_{\mscr S} f(s) d\mu(s) =  \int_{z\in \sigma} z E(dz) \in \mscr L(\F) .
	\end{equation}
Since $f$ is bounded, the spectrum of $U$ is contained in a compact set in the complex plane.

\begin{defn}\label{defn:isolated-cirlces}
We say that $U$ has a point spectrum concentrated on isolated circles if every point of the set
	\begin{equation*}
	\setc{ r \in \R }{ \exists s\in \mscr S, \abs{ f(s) } = r, (f(s)I - U)^{-1} \text{ does not exist}}
	\end{equation*}
is an isolated point.
\end{defn}

Equivalently, if $U$ has a point spectrum concentrated on isolated circles, then $\abs{\sigma_{p}(U)}$ is a discrete set.  This definition does not exclude the situation when eigenvalues are dense on a circle.  Such a situation arises naturally for dynamical systems possessing a limit cycle, as will be seen later.

\subsection{GLA for dominating point spectrum.}

\begin{lem}\label{lem:peripheral-equicontinuous-family}
Assume the peripheral spectrum of $U$ contains an eigenvalue $\lam \in \sigma_{p}(U)$.  Let $g_{n} : \sigma(U) \to \C$ be $g_{n}(z) = \lam^{-n} z^{n}$.  Then $\set{ g_{n}(U) }_{n\in \N} = \set{ \lam^{-n} U^{n} }_{n\in \N}$ is an equicontinuous family of operators.
\end{lem}

\begin{proof}
Since $f$ is bounded, then $\sigma$ is contained in a compact set in $\C$.  For each $n$, the support of $g_{n}$ is a compact set, $g_{n}$ is a continuous and, therefore, a bounded Baire function on $\sigma$.  Since $\lam$ is a peripheral eigenvalue, then $\abs{g_{n}(z) } \leq 1$ for all $z\in \sigma$ since by definition $\abs{z}\leq \abs{\lam}$ for all $z \in \sigma$.  If $f:\mscr S \to \C$ is the function associated with $U$, then $\abs{(g_{n} \circ f)(s)} \leq 1$ for all $n \in \N$ and $s \in \mscr S$.  By proposition \ref{prop:equicontinuous-family}, $\set{g_{n}(U) = \int_{\mscr S} (g_{n} \circ f) d\mu }_{n\in \N}$ is an equicontinuous family.  By the functional calculus, $g_{n}(U) = \lam^{-n} U^{n}$.
\end{proof}

\begin{prop}\label{prop:peripheral-projection}
Assume the peripheral spectrum of $U$ contains an eigenvalue $\lam \in \sigma_{p}(U)$.  Then
	\begin{equation}
	A_{n}(\lam^{-1}U) := n^{-1} \sum_{k=1}^{n} \lam^{-k} U^{k}
	\end{equation}
converges to the projection operator $P_{\lam} := E(\set{\lam})$ onto $N(\lam I - U)$ in the strong operator topology; that is, for every continuous seminorm $p$ and every element $\psi \in \F$
	\begin{equation*}
	\lim_{n\to \infty} p( A_{n}(\lam^{-1}U) \psi - P_{\lam} \psi) = 0.
	\end{equation*}
Furthermore, the projection $P_{\lam}$ commutes with $U$ and we have the direct sum decomposition $\F = N(\lam I - U) \oplus \overline{ \Image(\lam I - U)}$.
\end{prop}

\begin{proof}
By lemma \ref{lem:peripheral-equicontinuous-family}, $\set{ (\lam^{-1}U)^{n} }$ is an equicontinuous family of operators on a locally sequentially weakly compact LCLT space since $(\lam^{-1}U)^{n} = g_{n}(\lam^{-1}U)$, where $g_{n}$ is the function in the above lemma.  By Yosida's mean ergodic theorem, the averages of this family converge to a projection operator onto the $\lam$-eigenspace.  Since projections take the form of indicator functions on subsets of $\sigma \subset \C$ (lemma \ref{eq:spectral-projections-lct}), then this projection takes the form $E(\set{\lam})$.
\end{proof}

\begin{rem}
Proposition \ref{prop:peripheral-projection} is true for any other eigenvalue in the peripheral spectrum, say $\lam' \in \sigma_{p}$.  For this $\lam'$, the proposition can be applied to obtain the projection onto $N(\lam' I - U)$.  Additionally, $N(\lam' I - U) \subset \overline{\Image(\lam I - U)} $.  Since $\overline{\Image(\lam I - U)}$ is $U$-invariant and $\F$ is locally sequentially weakly compact, then $(\lam')^{-n}(U' )^{n}$ is a equicontinuous family of operators on the locally sequentially weakly compact space $\overline{\Image(\lam I - U)}$, where $U'$ is a the restriction of $U$ to $\overline{\Image(\lam I - U)}$.  By proposition \ref{prop:peripheral-projection}, $\overline{\Image(\lam I - U)} = N(\lam' I - U') \oplus \overline{\Image(\lam I - U')}$.  Since $N(\lam' I - U') = N(\lam' I - U)$ and $\Image(\lam' I - U') = \Image((\lam' I - U)(I - P_{\lam}))$, then
	\begin{equation*}
	\F = N(\lam I - U) \oplus N(\lam' I - U) \oplus \overline{\Image((\lam' I - U)(I - P_{\lam}))} .
	\end{equation*}
\end{rem}

\begin{cor}\label{cor:peripheral-direct-sum-decomposition}
Let the peripheral spectrum of $U$ consist of only (at most countably many) eigenvalues.  Then
	\begin{equation}\label{eq:peripheral-direct-sum}
	\F = \left[ \bigoplus_{j\geq 1} N(\lam_{j} I - U) \right] \oplus \F'
	\end{equation}
where  $\set{ \lam_{j} }$ is some labeling of the peripheral spectrum and $\F'$ is the subspace of $\F$ corresponding to the part of the spectrum not in the peripheral spectrum.  The subspace $\F'$ has the form
	\begin{equation}\label{eq:residual-subspace}
	\F' = \overline{ \left(I - { \sum\limits_{j\geq 1} } E(\set{\lam_{j}}) \right)} .
	\end{equation}
\end{cor}

The above results can be combined to give a recursive procedure for computing projections onto eigenspaces for any $\psi \in \F$.  In order to do this, we must compute the projection onto eigenspaces corresponding to eigenvalues of largest modulus first, subtract these from $\psi$, then compute the projections for eigenvalues of the next largest modulus.

\begin{thrm}[Generalized Laplace Analysis]\label{eq:gla-lcs-forward}
Let $\sigma(U)$ have dominating point spectrum and assume that the point spectrum is concentrated on isolated circles in the complex plane (definition \ref{defn:isolated-cirlces}).  Let $\lam$ be an eigenvalue for $U$.  Denote $E(\set{\lam})$ by $P_{\lam}$.  Then
	\begin{equation}\label{eq:gla-lcs}
	\begin{aligned}
	P_{\lam} &=  \lim_{n\to \infty} A_{n}\left( \lam^{-1} U(I - \sum_{ \mu \in \Omega} P_{\mu} ) \right) 
	=  \lim_{n\to \infty} n^{-1} \sum_{k=1}^{n-1} \lam^{-n} U^{n} \left( I - \sum_{ \mu \in \Omega } P_{\mu} \right)
	\end{aligned}
	\end{equation}
exists in the strong topology and where $\Omega := \setc{\mu \in \sigma_{p}(U) }{\abs{\mu} > \abs{\lam}}$.  Furthermore, if $\Omega$ is a finite set, then $P_{\lam} \psi$ can be obtained via a finite recursion processes by computing $P_{\mu} \psi$ with \eqref{eq:gla-lcs} for each $\mu \in \Omega$ and subtracting it from $\psi$.
\end{thrm}


\begin{proof}
Fix $\lam \in \sigma_{p}(U)$ and let $\Omega = \setc{ \mu \in \sigma(U) }{ \abs{\mu} > \abs{\lam} }$.  Since $\sigma(U)$ has dominating point spectrum, $\Omega$ consists only of eigenvalues.  By property \eqref{item:dominating-subset} of lemma \ref{lem:dominating-spectrum-properties}, $\sigma(U) \setminus \Omega$ has dominating point spectrum.  Since the point spectrum of $U$ is concentrated on isolated circles, $h(z) = 1 - \ind{\Omega}(z)$ is a bounded, continuous function on $\sigma(U)$ and $h(U) = \int_{\sigma(U)} h(z) E(dz) =  I - E(\Omega)$.  By the spectral mapping theorem with the continuous function $z \mapsto zh(z)$, $\sigma( U (I - E(\Omega)) ) = (\sigma(U) \setminus \Omega) \cup \set{0}$.  Since $\sigma(U) \setminus \Omega$ has dominating point spectrum, so does $\sigma( U (I - E(\Omega)) )$.  Since $\lam \in \sigma( U (I - E(\Omega)) )$ and $\infnorm{ U (I - E(\Omega)) } \leq \abs{\lam}$, then $\lam$ is a peripheral eigenvalue for $U(I - E(\Omega))$.

Let $g_{n}(z) = \lam^{-n} z^{n} (1 - \ind{\Omega}(z))$.  Then $g_{n}$ is a continuous function satisfying $\abs{g_{n}(z)} \leq 1$ for all $n\in\N$ and $z \in \sigma$ since it is 0 on $\Omega$ and $\lam$ is a peripheral eigenvalue in $\sigma( U ) \setminus \Omega$.  Therefore, $g_{n}$ is a Baire function for every $n$.  If $f :\mscr S \to \C$ is the function associated with $U$, then $\abs{ (g_{n} \circ f)(s)} \leq 1$ for all $n\in \N$ and $s\in \mscr S$.  Therefore, by proposition \ref{prop:equicontinuous-family}, $\set{g_{n}(U)}_{n \in \N}$ is an equicontinuous family.

By the multiplicative property of the spectral integrals
	\begin{align*}
	\left[\lam^{-1} U (I - E(\Omega)) \right]^{n} 
	&= \int_{\sigma} \lam^{-n} z^{n}(1 - \ind{\Omega}(z))^{n} E(dz) \\
	&= \left( \int_{\sigma} \lam^{-n} z^{n} E(dz) \right) \cdot \left( \int_{\sigma} (1 - \ind{\Omega}(z)) E(dz) \right) \\
	&= \lam^{-n}U^{n} \left( I - E(\Omega) \right) .
	\end{align*}
Since the point spectrum is at most countable and $\Omega$ contains only eigenvalues, then $\Omega$ is a countable union of singleton sets of eigenvalues,  $\Omega = {\textstyle \bigcup\limits_{\mu \in \Omega} \set{\mu} } $.  Since the spectral measure is countably additive, then $E(\Omega) = \sum\limits_{\mu \in \Omega} E(\set{\mu})$.  Therefore, $\lam^{-n}U^{n} \left( I - E(\Omega) \right) = \mu^{-n} U^{n} ( I - \sum\limits_{\mu \in \Omega} E(\set{\mu})$.

By proposition \ref{prop:peripheral-projection}, the right side of \eqref{eq:gla-lcs} converge strongly to a $\lam^{-1}U ( I - \sum\limits_{\mu \in \Omega} E(\set{\mu}) )$-invariant function, say $\psi_{0} \in \F$.  Then $U ( I - \sum\limits_{\mu\in\Omega} E(\set{\mu}) ) \psi_{0} = \lam \psi_{0}$.  We claim that $\psi_{0}$ is in the nullspace of $E(\set{\mu})$ for every $\mu \in \Omega$.  Suppose this was not the case and $\psi_{0} \neq 0$.  Then there is some $\mu' \in \Omega$ such that $E(\set{\mu'}) \psi_{0} = \phi \neq 0$.  But $E(\set{\mu'})$ is the projection onto $N(\mu' I - U)$.  Therefore, we get $U( I - E(\set{\mu'}) ) \psi_{0} = U\psi_{0} - \mu' \psi_{0} = \lam \psi_{0}$.  Then $U \psi_{0} = (\mu' + \lam) \psi_{0}$, which implies that $\psi_{0} \in N((\mu'+\lam)I - U)$.  But since $\lam \neq 0$, this implies that $\psi_{0} \in N((\mu'+\lam)I - U) \cap N(\mu' I - U)$.  Consequently, $\psi_{0} = 0$, contrary to assumption.  Therefore, $\psi_{0}$ is in the nullspace of $\sum\limits_{\mu \in \Omega} E(\set{\mu})$ and we get that $U\psi_{0} = \lam \psi_{0}$.  Consequently, the averages converge strongly to the projection $E(\set{\lam})$ onto $N(\lam I - U)$. 
\end{proof}

\subsection{GLA for minimal point spectrum.}

The above results gave a procedure for constructing eigenfunctions when the point spectrum dominated the rest of the spectrum.  Unfortunately, this situation does not hold in a number of cases of interest.  Consider a dynamical system with an attractor.  In this setting, the spectrum on the unit circle corresponds to the attractor.  Since the system is asymptotically stable, eigenvalues corresponding to eigenfunctions supported off-attractor are contained strictly inside the unit circle.  In this system, the point spectrum may not dominate the spectrum since there may be continuous parts of the spectrum contained in the unit circle.  If we wish to project onto the off-attractor, stable eigenspaces, we need to modify the above GLA procedure which was valid in the presence of a dominating point spectrum.  The general idea is to consider the inverse operator $U^{-1}$.  If $U$ has point spectrum inside the unit circle, then $U^{-1}$ has point spectrum outside the unit circle.  The GLA theorems of the last section can then be applied to $U^{-1}$ to obtain projections onto the stable directions of the attractor.  Proposition \ref{prop:gla-inverse-iteration} formalizes this.

\begin{defn}[Minimal point spectrum]
Let $U \in \mscr A$.  We say that $\sigma(U)$ has a minimal point spectrum if $\sigma^{-1}(U) := \setc{ \lam^{-1} }{\lam \in \sigma(U)} \in \C \cup \set{\infty}$ has a dominating point spectrum.
\end{defn}

\begin{prop}[Inverse GLA]\label{prop:gla-inverse-iteration}
Let $U = \int_{\sigma} z E(dz) \in \mscr A$ have a minimal point spectrum.  Additionally, assume that the point spectrum is concentrated on isolated circles and the spectrum satisfies $0 < C^{-1} \leq \inf\limits_{\xi\in\sigma(U)} \abs{\xi}$.  Then for $\lam \in \sigma_{p}(U)$,
	\begin{equation}
	E(\set{\lam}) = \lim_{n\to \infty} A_{n}\left( \lam U^{-1}( I - E(\Omega) ) \right)
	\end{equation}
where $\Omega = \setc{ \xi \in \sigma(U) }{ \abs{\xi} < \abs{\lam} }$.
\end{prop}

\begin{proof}
Since $0 < C^{-1} \leq \inf\limits_{\xi\in\sigma(U)} \abs{\xi}$, then $i(z) = z^{-1}$ is continuous and bounded on $\sigma(U)$.  By the multiplicative property of the integral
	\begin{equation*}
	U\cdot i(U) = \int_{\sigma(U)} z \cdot i(z) E(dz) = I = \int_{\sigma(U)} i(z) \cdot z  E(dz) = i(U)\cdot U .
	\end{equation*}
Therefore $i(U) = U^{-1}$. By the spectral mapping theorem $\sigma(U^{-1}) = \setc{ \xi^{-1} }{ \xi \in \sigma(U)}$.  Therefore, $U^{-1}$ has a bounded, dominating point spectrum that is concentrated on isolated circles.  By the change of measure $E_{1} := E \circ i^{-1}$, we have the representation of $U^{-1}$ as the integral
	\begin{equation*}
	U^{-1} = \int_{\sigma(U^{-1})} w E_{1}(dw).
	\end{equation*}

Fix $\lam \in \sigma_{p}(U)$ and let $\Omega' = \setc{\xi^{-1} \in \sigma(U^{-1})}{ \abs{\xi^{-1}} > \abs{\lam^{-1}}}$.  By proposition \ref{eq:gla-lcs-forward}, 
	\begin{equation*}
	E_{1}(\set{\lam^{-1}}) = \lim_{n\to\infty} A_{n}\left( (\lam^{-1})^{-1} U^{-1}\left( I - \sum_{\xi^{-1} \in \Omega'} E_{1}(\set{\xi^{-1}})	\right)	\right)
	\end{equation*}
with convergence in the strong topology.  Since $N(\lam I - U) = N(\lam^{-1} I - U^{-1})$, then $E(\set{\lam}) = E_{1}(\set{\lam^{-1}})$ for all $\lam \in \sigma_{p}(U)$.  Consequently,
	\begin{equation*}
	E(\set{\lam}) = \lim_{n\to\infty} A_{n}\left( \lam U^{-1}\left( I - \sum_{\xi \in \Omega} E(\set{\xi})	\right)	\right) = \lim_{n\to\infty} A_{n}\left( \lam U^{-1}\left( I - E(\Omega)	\right)	\right)
	\end{equation*}
where $\Omega = \setc{\xi \in \sigma(U)}{ \abs{\xi} < \abs{\lam}}$.
\end{proof}

\begin{rem}[GLA for strongly continuous groups]
The GLA theorems can be extended to hold for strongly continuous groups of spectral operators on locally  sequentially weakly compact Banach spaces.  In this setting, the Laplace averages of the group take the form
	\begin{equation}\label{eq:continuous-time-average}
	A_{\alpha}(\lam^{-1}G(\cdot)) f := \alpha^{-1} \int_{0}^{\alpha} \lam^{-t} G(t) f .
	\end{equation}
The extension of the GLA theorems in this case rest upon a version of Yosida's mean ergodic theorem for strongly continuous semigroups of operators $\set{ G(t) }$, $t \in \R^{+}$, that roughly says that if $n^{-1} G(n) f$ converges to 0 and the averages of the semigroup do not behave too badly between integer times, then the limit of the continuous time averages $\al^{-1} \int_{0}^{\al} G(t)$ converges in the strong operator topology as $\al \to \infty$ to a projection operator on the subspace of elements that are fixed points of the semigroup (see ch.\ VIII.7, thm.\ 1 of \cite{Dunford:1958wx}, and its corollary, for a precise statement).  A second result used in the extension is due to Lange and Nagy \cite{Lange:1994gk} on the representation of strongly continuous \emph{groups} of scalar-type operators with spectrums contained in the unit circle as a spectral integral over a common spectral measure; i.e., the group of operators has a representation as $G(t) = \int_{\R} e^{it\lam} E(d\lam)$.  A sketch of the proof is as follows.

If we assume that there is a set of isolated circles in the complex plane and that each $G(t)$ in the group of operators is spectral with point spectrum restricted to these circles and furthermore that every operator from the group has a dominating point spectrum with the same $R$ (see definition \ref{defn:dominating-point-spectrum}), then we can scale the family by an eigenvalue $\lam = \rho_{1}e^{i\omega}$, where $\rho_{1}$ is the radius of the largest circle containing the spectrum, and construct a new family of spectral operators $G_{1}(t) := (\rho_{1}e^{i\omega})^{-t}G(t)$.  This new family will have part of the spectrum contained in the unit circle and the rest strictly inside.  This new group can be split into two groups, the unitary part $U_{1}(t)$, with spectrum contained in the unit circle, and the dissipative part $D_{1}(t)$, with spectrum contained strictly inside and bounded away from the unit circle uniformly in $t$.  Then $(\rho_{1}e^{i\omega})^{-t}G(t) = G_{1}(t) = U_{1}(t) + D_{1}(t)$.  Lange and Nagy's representation of the unitary part allows us to write $U_{1}(t) = \int_{\R} e^{it\lam} E_{1}(d\lam)$ and the above extension of Yosida's mean ergodic theorem allows us to prove that averages of $U_{1}(t)$ converge to the projection operators onto the subspace of $U_{1}(t)$-invariant elements.  An average of the dissipative part converges to zero with order $O(\alpha^{-1})$ since the norm of this average has order $\alpha^{-1} \int_{0}^{\alpha} O(e^{-\beta t}) dt$.  Here, $\beta > 0$ is related to the gap between the unit circle and the largest circle contained in $\D$ on which the spectrum of $D_{1}(t)$ is concentrated.  Combined, the above arguments give that the Laplace averages $A_{\al}( e^{-i\omega} \rho_{1}^{-1} G(\cdot) ) = \al^{-1} \int_{0}^{\al} (\rho_{1}e^{i\omega })^{-t} G(t) dt$ converge in the strong operator topology to the subspace of elements such that $G(t)\psi = (\rho_{1}e^{i\omega})^{t} \psi $.
\end{rem}

\section{Connections with Dynamical Systems and the Koopman Operator}\label{sec:dyn-sys}
This section connects the abstract results of the previous section with dynamical systems.  The operator we are interested in is the Koopman operator associated with a dynamical system $(\X,\Phi)$, where $\Phi: \X \to \X$ (we have yet to put any structure on $\Phi$).  Recall that the Koopman operator is the composition operator $U_{\Phi}\psi = \psi \circ \Phi$, for all $\psi$ in some space of functions $\F$.  In particular, we will show that the conditions on the set of eigenvalues $\Lambda$ (namely, it being a bounded set and concentrated on isolated circles) are natural and are satisfied for the examples we consider.  Additionally, we will construct a space of observables for both attracting fixed points and limit cycles on which the induced Koopman operator is spectral.  Connections are drawn between these spaces and the Hilbert Hardy space $H^{2}(\D)$.

For problems on the attractor, it is natural to consider the space $L^{2}(\nu)$, where $\nu$ a probability measure supported on the attractor that is preserved by the transformation.  This case has been treated extensively in the literature.  However, this space is not particularly natural for dissipative dynamics.  For example, consider a dynamical system on the real line with the origin as a globally attracting fixed point.  The invariant measure for this system is a delta measure supported at zero,  in which case, the associated Hilbert space of observables, $L^{2}(\R,\delta_{0})$, is isomorphic to $\R$; every function agreeing at zero is equivalent regardless of their values away from zero.  The space $L^{2}(\R,\delta_{0})$ cannot give any information about the dynamics away from the attractor.  Natural observables (as seen in section \ref{subsubsec:stable-fixed-point} below) for the dissipative dynamics are polynomials and their completions under certain norms.

The construction of spaces on which the Koopman operator is spectral is accomplished as follows.  First we consider the the linearized system $(\hat{\mscr Y}, \hat{\msf A})$ and the associated Koopman operator $U_{\hat{\msf A}}$, where $\hat{\mscr Y}$ is contained in the open unit cube $Q_{1}$ and $\hat{\msf A}$ is the linearization of the dynamics $\Phi$ around the attractor.  We find a set of eigenfunctions for $U_{\hat{\msf A}}$ such that arbitrary pointwise products from this set are also well-defined functions on $\hat{\mscr Y}$.  These products will also be eigenfunctions of the Koopman operator (see prop.\ 5, \cite{Budisic:2012cf}).  It will be shown that these products will generate a space of polynomials over a normed unital commutative rings.  This space of polynomials will be completed to a Banach space of observables $\F$ using an $\ell^{2}$ polynomial norm, as in lemma \ref{lem:polynomial-completion} below.  

Construction of observables for the nonlinear dynamical system can be constructed through a conjugacy map and a pullback construction.  Assume $(\X, \Phi)$ is topologically conjugate to the linearized system $({\mscr Y}, \msf A)$ via the diffeomorphism $h : \X \to {\mscr Y}$; $h \circ \Phi = \msf A \circ g$.  Let $B \subset \X$ be a simply connected, bounded, positively invariant open set in $\X$ such that $h(B) \subset Q_{r} \subset \mscr Y$, where $Q_{r}$ is a cube in $\mscr Y$.  Scaling $Q_{r}$ to the unit cube $Q_{1}$ via the smooth diffeomorphism $g : Q_{r} \to Q_{1}$ gives $(g\circ h)(B) \subset Q_{1}$.  Then if $\psi \in \F$ is an eigenfunction for $U_{\hat{\msf A}} : \F \to \F$ at $\lam$, then $\psi \circ g \circ h$ is an eigenfunction for $U_{\Phi}$ at eigenvalue $\lam$ (see prop.\ 7, \cite{Budisic:2012cf}).  The observable space for $U_{\Phi}$ will be given as $\F\circ g \circ h = \setc{ \phi \circ g \circ h}{ \phi \in \F}$.  See figure \ref{fig:conjugacy-chain} to for a schematic.  Results from Lan and Mezi\'{c} \cite{Lan:2012vw} guarantees the existence of a $C^{1}$-topological conjugacy between the nonlinear and linear system which is valid in the entire basin of attraction of certain hyperbolic attractors as long as the nonlinear dynamical system $(\X, \Phi)$ is sufficiently smooth.

\begin{figure}[h]
\begin{center}
\includegraphics[width=0.95\textwidth]{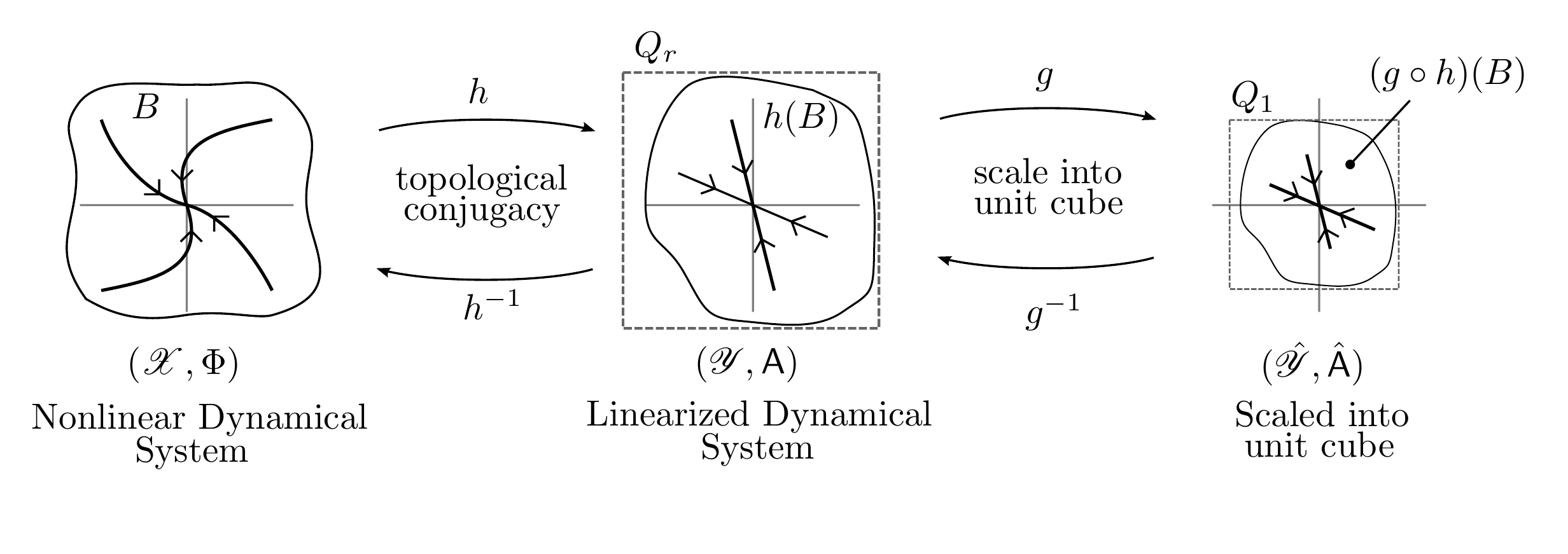}
\caption{Chain of topological conjugacies used to construct eigenfunctions in the basin of attraction $B$ of the fixed point for the Koopman operator corresponding to the nonlinear system.  The existence if $h : B \to h(B)$ is guaranteed by a theorem in \cite{Lan:2012vw}. }
\label{fig:conjugacy-chain}
\end{center}
\end{figure}

\subsection{Polynomials Over Normed Commutative Rings.}
Let $(\mscr R,+,\cdot)$ be a normed unital commutative ring with the norm $\norm{\cdot}$.

\begin{defn}
The space of polynomials over the ring $\mscr R$ in the indeterminates $x_{1},\dots, x_{d}$ is defined by
	\begin{equation}\label{eq:polynomial-over-normd-ring}
	\mscr R[x] := \setc{ \sum_{\abs{k} \leq K} \psi_{k}\cdot x^{k} }{ K \in \N_{0}, \psi_{k} \in \mscr R},
	\end{equation}
where $x = (x_{1}, \dots, x_{d})$, $k = (k_{1},\dots, k_{d}) \in \N_{0}^{d}$, $\abs{k} := \sum\limits_{j=1}^{d} \abs{k_{j}}$, and $x^{k} := x_{1}^{k_{1}}\cdots x_{d}^{k_{d}}$.
\end{defn}
These are just regular polynomials with coefficients taking values in $\mscr R$ rather than $\R$ or $\C$.  The space $\mscr R[x]$ is a normed unital commutative ring under the normal operations of addition and convolution products for polynomials.  The space $\mscr R[x]$ is isomorphic to the sequence space $c_{00}(\mscr R^{\N_{0}^{d}})  = \setc{ \set{\psi_{k}}_{k \in \N_{0}^{d} } }{ \psi_{k} \in \mscr R, F\subset \N_{0}^{d}\text{ finite}, \forall k \in F^{c}, \psi_{k} = 0}$ under the ring isomorphism $i : \mscr R[x] \to c_{00}(\mscr R^{N_{0}^{d}})$ given by $i( \sum_{k \in \N_{0}^{d}} \psi_{k}\cdot x^{k} ) = \set{\psi_{k}}_{k \in \N_{0}^{d}}$.  Only finitely many terms of the sequence are nonzero.  The norm $\norm{\cdot}$ on $\mscr R$ induces a norm on $\mscr R[x]$:
	\begin{equation}\label{eq:polynomial-norm}
	\norm{  \sum_{k \in \N_{0}^{d}} \psi_{k} x^{k} }_{\mscr R,2} =  ( \sum_{k \in \N_{0}^{d}} \norm{ \psi_{k} }^{2} )^{1/2}.
	\end{equation}
Expression \eqref{eq:polynomial-norm} gives a norm for the sequence space as well.  Under this norm, the isomorphism $i$ is an isometric isomorphism.
Define the coordinate projections $\pi_{m} : \mscr R [x] \to \mscr R$ by
	\begin{equation}\label{eq:coordinate-projections}
	\pi_{m} \left( \sum_{k \in \N_{0}^{d}} \psi_{k} x^{k}  \right) = \psi_{m}, \qquad (\forall m \in \N_{0}^{d}) .
	\end{equation}

\begin{defn}
Let $\B$ be a separable reflexive Banach space under the norm $\norm{\cdot}$.  Let
	\begin{equation}
	\ell^{2}\left(\B^{\N_{0}^{d}} \right) 
	:= \setc{ \set{\psi_{k}}_{k \in \N_{0}^{d} } }{ \psi_{k} \in \B, \sum_{k \in \N_{0}^{d} } \norm{\psi_{k}}^{2} < \infty} .
	\end{equation}
\end{defn}

\begin{defn}[$\B$-Hardy space]\label{def:bhardy-space}
If $\B$ a separable reflexive Banach space under the norm $\norm{\cdot}$ and $(\mscr R,+)$ is a dense subspace of $\B$, denote the completion of $\mscr R[x]$ under the $\ell^{2}$ polynomial norm \eqref{eq:polynomial-norm} by $\bhardy$.  The completion is 
	\begin{equation}\label{eq:bhardy-space}
	\bhardy = \setc{ \sum_{k \in \N_{0}^{d}} \psi_{k} x^{k} }{  \psi_{k} \in \B, \sum_{k \in \N_{0}^{d}} \norm{\psi_{k}}^{2} < \infty} .
	\end{equation}
\end{defn}

\begin{lem}\label{lem:polynomial-completion}
Let $(\mscr R,+)$ be a dense subspace of a separable reflexive Banach space $\B$ under the norm $\norm{\cdot}$.  Then $\bhardy$ is isometrically isomorphic to $\ell^{2}\left(\B^{\N_{0}^{d}} \right)$.
\end{lem}

\begin{proof}
Let $c_{00}(\B^{\N_{0}^{d}})$ be the subspace of $\ell^{2}\left(\B^{\N_{0}^{d}} \right)$ consisting of elements having only finitely many nonzero terms.  Since $(\mscr R,+)$ is dense in $\B$, then $\mscr R[x]$ is isometrically isomorphic to a subspace dense in $c_{00}(\B^{\N_{0}^{d}})$ .  Since $c_{00}(\B^{\N_{0}^{d}})$ is dense in $\ell^{2}\left(\B^{\N_{0}^{d}} \right)$, then $\mscr R[x]$ is also isometrically isomorphic to a subspace dense in $\ell^{2}\left(\B^{\N_{0}^{d}} \right)$.
\end{proof}

\begin{prop}\label{prop:lswc}
Let $(\mscr R, +)$ be a dense subspace of a separable, reflexive Banach space.  Then \bhardy is locally sequentially weakly compact.
\end{prop}

\begin{proof}
Since $\B$ is separable and reflexive, so is \bhardy (\cite{Wojtaszczyk1991}, p.44).  Fix $\psi \in \bhardy$ and consider the closed neighborhood $U = \setc{ \phi \in \bhardy }{ \norm{\phi - \psi}_{\mscr R,2} \leq \eps}$.  This is a bounded neighborhood of $\psi$.  Since \bhardy is both a separable and reflexive Banach space, the sequential version of the Banach-Alaoglu theorem implies that this neighborhood is sequentially weakly compact.  Since this is true for every $\psi$, then \bhardy is locally sequentially weakly compact.
\end{proof}

In this paper, $\B$ will be a separable reflexive Banach space of functions whose domain is the attractor of the dynamical system.  The normed ring $\mscr R$ will be a dense subspace of $\B$ where the ring multiplication is  given by pointwise products of functions.

We should remark that while \bhardy is a well-defined Banach space, elements of it are not necessarily well-defined observables on $\X$ even though all polynomials in $\mscr R[x]$ are.  For example, let $\X = \R^{d}$.  Consider a point of $\R^{d}$ having at least one coordinate $x_{i}$ satisfying $\abs{ x_{i} } > 1$.  Then there are infinitely many elements of \bhardy such that the infinite series diverges if we substitute the above point in for the indeterminates.  Hence, not all elements of \bhardy define finite-valued functions on $\R^{d}$.  On the other hand, if $\X$ is contained strictly in the open unit cube of $\R^{d}$, then every series in \bhardy is convergent, since $\abs{x_{i}} < 1$ for every point $x = (x_{1},\dots, x_{d})$.   If we restrict our attention to bounded subsets $B$ in $\R^{d}$ containing the origin, by a change of variables $g$ which maps $B$ into a set $g(B)$ that is contained strictly in the unit cube, we get a well-defined space of observables $\bhardy \circ g$, where \bhardy is defined on the open unit cube.

\subsection{Spaces of Observables for Bounded attractors}\label{sec:examples}

\subsubsection{Asymptotically stable fixed points in $\mathbb{C}^{d}$.}\label{subsubsec:stable-fixed-point.}

We start with an example of a diagonalizable linear system.

\begin{example}[Stable diagonalizable linear system]\label{ex:generated-polynomials-over-Banach-algebra}
Let $\mscr Y = \C^d$ with the Euclidean norm and $\msf A$ a $d\times d$ diagonalizable matrix.  Define the dynamics as $\mbf y_{k+1} = \msf A \mbf y_{k}$.  Let $\set{\mbf v_1, \dots, \mbf v_d}$ be a basis of eigenvectors of $\msf A$ corresponding to nonzero eigenvalues $\set{\lambda_j}_{1}^{d}$.  Then $\mbf y = \sum_{1}^{d} c_j \mbf v_j$.  

Let $\set{ \mbf w_{j} }_{j=1}^{d}$ be the adjoint basis to $\set{ \mbf v_{j}}_{j=1}^{d}$; then $\inner{\mbf v_{j}}{\mbf w_{k}} = \delta_{jk}$ and $\mbf w_{j}$ is an eigenvector of $\msf A^{*}$ at eigenvalue $\bar{\lam}_{j}$.  Define observables by the linear functionals
	\begin{equation}\label{eq:principle-eigenfunctions-linear-dynamics}
	\phi_j(\mbf y) = \inner{\mbf y}{\mbf w_j}
	\end{equation}
for all $\mbf y$ and $j=1,\dots, d$.  The observable $\phi_j$ is a nonzero eigenfunction of $U_{\msf A}$ since
	\begin{align*}
	(U_{\msf A}\phi_j)(\mbf y) = \phi_j(\msf A\mbf y) &= \inner{\msf A\mbf y}{\mbf w_j} 
	= \inner{\mbf y}{\msf A^{*}\mbf w_j} 
	= \inner{\mbf y}{\bar{\lambda}_j \mbf w_j} 
	= \lambda_j \phi_j(\mbf y).
	\end{align*}
Additionally, the product $(\phi_1 \cdots \phi_d)$ is not equivalently the zero functional.  Therefore, the semigroup property for eigenfunctions (proposition 5, \cite{Budisic:2012cf}) implies that for any $(m_1,\dots, m_d) \subset \N_0^d$,
	\begin{equation}\label{eq:eigenpairs-linearized-dynamics}
	\left(\prod_{j=1}^{d} \lambda_{j}^{m_j}, \, \prod_{j=1}^{d} \phi_{j}^{m_j} \right)
	\end{equation}
is an eigenpair for the corresponding Koopman operator.

Let $\mc G = \setc{ \phi_j : \mscr Y \to \C }{ \forall  j=1,\dots, d }$, be called the principle eigenfunctions of $U_{\msf A}$ defined in \eqref{eq:principle-eigenfunctions-linear-dynamics} above.  Define the sets of functions 
	\begin{equation*}
	\mscr P_{\mc G} := \setc{ c : \mscr Y \to \C }{ c(\mbf y) = \prod_{j=1}^{d} \phi_{j}^{m_{j}}(\mbf y), (m_{1},\dots,m_{d}) \subset \N_{0}^{d} }
	\end{equation*}
and
	\begin{equation}\label{eq:fixed-point-ring}
	\mscr R[\hat{\mbf y}] := \setc{ \sum_{i=1}^{n} \alpha_{i} c_{i}(\mbf y) }{ n \in \N, \alpha_{i} \in \C, c_{i} \in \mscr P_{\mc G} }.
	\end{equation}
Then, $\mscr R[\hat{\mbf y}]$, where $\hat{\mbf y} = ( \phi_{1}(\mbf y), \dots, \phi_{d}(\mbf y))$, is the space of polynomials over the normed unital commutative ring $L^{2}(\C^{d},\delta_{0})$, where $\delta_{0}$ is the Dirac measure supported at $\mbf y=0$.  The ring $\mscr R[\hat{\mbf y}]$ is isomorphic to $\C[\hat{\mbf y}]$.  Furthermore, $\mscr R[\hat{\mbf y}]$ is the space of finite linear combinations of eigenfunctions of the Koopman operator corresponding to the dynamical system $(\mscr Y, \msf A)$.

To see this, put $\msf V = [\mbf v_{1}, \cdots, \mbf v_{d}]$ and $\msf W = [\mbf w_{1}, \cdots, \mbf w_{d}]$.  Define the new coordinates $\hat{\mbf y} = [\hat y_{1}, \dots, \hat y_{d}]^{\msf T}$ by the mapping $\hat{\mbf y} = g( \mbf y ) = \msf W^{*} \mbf y$, so that, in particular, $\hat y_{j} = \phi_{j}(\mbf y)$.  This is invertible since $\msf W$ is and the inverse is given by $\mbf y = g^{-1}(\hat{\mbf y}) = \msf V \hat{\mbf y}$.  By definition of the functions $c_{i} \in \mscr P_{\mc G}$,  $c_{i}(\mbf y) = \prod_{j=1}^{d} \phi_{j}^{m_{j,i}}(\mbf y) = \prod_{j=1}^{d} \hat{y}_{j}^{m_{j,i}} = \hat{\mbf y}^{\mbf m_{i}}$ for $\mbf m_{i} \in \N_{0}^{d}$.  Since $L^{2}(\C^{d},\delta_{0}) \cong \C$, then every element of $\mscr R[\hat{\mbf y}]$ has the form $\sum_{i=1}^{n} \alpha_{i} \hat{\mbf y}^{\mbf m_{i}} $ which is a polynomial in the indeterminates $\hat y_{1}, \dots, \hat y_{d}$ with coefficient in $\C \cong L^{2}(\C^{d},\delta)$.  Therefore, $\mscr R[\hat{\mbf y}]$ is the space of polynomials over the normed unital commutative ring $L^{2}(\delta_{0})$ and $\mscr R[\hat{\mbf y}] \cong \C[\mbf{\hat{y}}]$.
The final result follows directly from $\mscr P_{\mc G}$ consisting of eigenfunctions of the Koopman operator and the definition of $\mscr R[\hat{\mbf y}]$.
\end{example}

\begin{prop}
If $\mscr R[\mbf{ \hat y}]$ is the ring of polynomials \eqref{eq:fixed-point-ring}, then the associated Koopman operator $U_{\msf A}$ leaves $\mscr R_{K}[\mbf{ \hat y}]$ invariant for all $K \in \N_{0}$.
\end{prop}

\begin{proof}
Recall $\mscr R_{K}[\hat{ \mbf y}]$ is the space of polynomials having multi-degree modulus no greater than $K$.  This space has the form
	\begin{equation*}
	\mscr R_{K}[\hat{ \mbf y}] = \set{ \sum_{\abs{\mbf k}\leq K} \psi_{\mbf k} (\hat{y}_{1}^{k_{1}}\cdots \hat{y}^{k_{d}}) } = \set{ \sum_{\abs{\mbf k}\leq K} \psi_{\mbf k} \left(\phi_{1}^{k_{1}}\cdots \phi_{d}^{k_{d}}\right)(\mbf y) },
	\end{equation*}
where $\phi_{j}$ are the eigenfunctions in \eqref{eq:principle-eigenfunctions-linear-dynamics} of the Koopman operator and $\psi_{\mbf k} \in \C$.  Then
	\begin{align*}
	U\left[\sum_{\abs{\mbf k}\leq K} \psi_{\mbf k} \left(\phi_{1}^{k_{1}}\cdots \phi_{d}^{k_{d}}\right)(\mbf y)\right]
	&= \sum_{\abs{\mbf k}\leq K} \psi_{\mbf k} \left(\lam_{1}^{k_{1}}\cdots \lam_{d}^{k_{d}}\right) \left(\phi_{1}^{k_{1}}\cdots \phi_{d}^{k_{d}}\right)(\mbf y) \\
	&= \sum_{\abs{\mbf k}\leq K} \hat{\psi}_{\mbf k} \left(\phi_{1}^{k_{1}}\cdots \phi_{d}^{k_{d}}\right)(\mbf y) \\
	&= \sum_{\abs{\mbf k}\leq K} \hat{\psi}_{\mbf k} (\hat{y}_{1}^{k_{1}}\cdots \hat{y}^{k_{d}}).
	\end{align*}
where $\hat{\psi}_{\mbf k} :=  \left(\lam_{1}^{k_{1}}\cdots \lam_{d}^{k_{d}}\right)\cdot \psi_{\mbf k}$ is in $\C$, since $\psi_{k} \in \C$.
\end{proof}

\begin{cor}\label{cor:U-spectral-linear-fixed-point}
$U_{\msf A} : \bhardy \to \bhardy$ is spectral, where \bhardy is the completion of the ring \eqref{eq:fixed-point-ring} under the norm \eqref{eq:polynomial-norm}.
\end{cor}

\paragraph{Construction of spaces observables of an asymptotically stable fixed point of a nonlinear dynamical system.}
Assume $\Phi(x) = \msf Ax + v(x)$, where $\msf A := D\Phi(0)$ is the linearization of $\Phi$ around the origin, $\msf A$ is diagonalizable and all its eigenvalues are contained strictly inside the unit circle, and $v :\R^{d} \to \R^{d}$ is a $C^{2}$ function.  Then $(\R^{d},\Phi)$ is a nonlinear dynamical system having an asymptotically stable fixed point at $0$.  Let $D\subset \R^{d}$ be the basin of attraction for the fixed point; it is possibly unbounded.  Let $(\R^{d}, \msf A)$ be the linearization of the dynamics around the fixed point.  Due to a theorem of Lan and Mezi\'{c} (thm.\ 2.4, \cite{Lan:2012vw}), there is a diffeomorphism $h : D \to h(D)$ under which the nonlinear and linear dynamics are topologically conjugate.  Let $B \subset D$ be a bounded, simply connected, positively invariant, open set containing the origin and having smooth boundary.  Then $h(B)$ is a bounded, simply connected, positively invariant open set containing the origin and having a $C^{1}$-boundary.  Since this set is bounded there is a $d$-dimensional open cube $Q_{r}$ containing $h(B)$.  There is a smooth diffeomorphism $g : \R^{d} \to \R^{d}$ mapping $Q_{r}$ onto the open unit cube $Q_{1}$ in $\R^{d}$ (namely just scaling each coordinate).  Since $(g \circ h)(B)$ is open, then every point $\hat{ \mbf y} \in (g \circ h)(B)$ has coordinates whose modulus is strictly smaller than 1 (see fig.\ \ref{fig:conjugacy-chain} to visualize these sets).

Let $\mscr R[\hat{\mbf y}]$ be the polynomials defined in example \ref{ex:generated-polynomials-over-Banach-algebra} with the indeterminates taking values in $D_{1} := (g \circ h)(B) \subset Q_{1}$.  Take the completion of this space, as in lemma \ref{lem:polynomial-completion}, and denote it as $\bhardy(D_{1})$.  Since all the coordinates have a modulus strictly smaller than 1, this space can be identified with a well-defined Banach space of functions having domain $D_{1} = (g \circ h)(B)$.  Using the conjugacies, we can define a Banach space of observables for the nonlinear dynamical system having domain $B$;
	\begin{equation}
	\F_{\Phi} := \bhardy \circ g \circ h := \setc{ (\psi \circ g \circ h)(\mbf x) }{ \psi \in \bhardy(D_{1}), \mbf x \in B }.
	\end{equation}
By proposition \ref{prop:lswc}, this space is locally sequentially weakly compact Banach space in addition to being separable.

\begin{prop}
Let $U_{\Phi} : \F_{\Phi} \to \F_{\Phi}$ be the Koopman operator associated with $\Phi$.  Assume $\Phi(x) = \msf Ax + v(x)$, where $\msf A$ is a diagonalizable matrix with eigenvalues $\set{\lam_{1},\dots, \lam_{d}}$ that satisfy $\abs{\lam_{i}} < 1$.  Then $U_{\Phi}$ is a spectral operator of scalar type and the spectrum of $U_{\Phi}$ is
	\begin{equation}
	\sigma(U_{\Phi}) = \setc{  \prod_{i=1}^{d} \lam_{i}^{k_{i}}  }{ \forall i \in \set{1,\dots, d},  k_{i} \in \N_{0} }.
	\end{equation}
\end{prop}

\begin{proof}
Elements of $\F_{\Phi}$ are infinite linear combinations of eigenfunctions of $U_{\Phi}$.  This follows from the construction of $\bhardy$ as infinite linear combinations of eigenfunctions of $U_{\msf A}$, that $\Phi$ and $\msf A$ are topologically conjugate in the entire basin of attraction by theorem 2.4 of \cite{Lan:2012vw}, and proposition 7 of \cite{Budisic:2012cf}.  The spectral measure $E : \sigma(U) \to \mscr L(\F_{\Phi})$ is given by
	\begin{equation}
	E(\Lambda)\psi = \sum_{ \setc{\mbf k \in \N_{0}^{d}}{ \mbf \lam^{\mbf k} \in \Lambda}  } \alpha_{\mbf k} \cdot \left[ (\phi_{1}\cdots \phi_{d})^{\mbf k} \circ g \circ h \right], \qquad (\psi \in \F_{\Phi}),
	\end{equation}
where $\Lambda \in \mc B(\C)$, $\psi = \sum\limits_{ {\mbf k \in \N_{0}^{d}} } \alpha_{\mbf k} \cdot \left[ (\phi_{1}\cdots \phi_{d})^{\mbf k} \circ g \circ h \right]$, and $(\phi_{1}\cdots \phi_{d})^{\mbf k} = (\phi_{1}^{k_{1}}\cdots \phi_{d}^{k_{d}})$ for $\mbf k = (k_{1},\dots, k_{d})$.  Since the norm on this space is given by the $\ell^{2}$-norm on coefficients, then its easy to show that $\norm{ E } \leq 1$, so that, in particular, the projections defined by the spectral measure form an equicontinuous family of operators.  It is straight forward to verify all the remaining properties spectral measures.
\end{proof}

Since $U_{\Phi} : \F_{\Phi} \to \F_{\Phi}$ is spectral and has dominating point spectrum concentrated on isolated rings, the GLA theorems can be applied in this space.

\begin{rem}\label{rem:hardy-space}
Put $d =1$ in example \ref{ex:generated-polynomials-over-Banach-algebra}.  It can be shown that $\phi(y) = y$ is an eigenfunction of the Koopman operator for $\lam = \msf A$.  This implies that $\phi_{k}(y) = y^{k}$ is also an eigenfunction at eigenvalue $\lam^{k}$.  Then elements of the completion $\bhardy$ of $\mscr R[\hat{\mbf{y}}]$, identified with a space of observables on $\D$, is 
	\begin{equation*}
	H^{2}_{\C}(\D) = \setc{ \sum_{k=0}^{\infty} \alpha_{k} y^{k} }{\alpha_{k} \in \C, \sum_{k\geq 0} \abs{\alpha_{k}}^{2} < \infty, y \in \D },
	\end{equation*}
since $\B = \C$.
This is the Hardy Hilbert space of observables $H^{2}(\D)$.  The spectrum of the Koopman operator on this space is
	\begin{equation}\label{eq:linearized-spectrum}
	\setc{ \lam^{k} }{ k = 0, 1, 2, \dots},
	\end{equation}
For a nonlinear dynamical system topologically conjugate to this linear one, we get as the pullback space of observables
	\begin{equation}\label{eq:d1-pullback-space}
	\F = H^{2}_{\C}(\D) \circ g \circ h,
	\end{equation}
where $g$ and $h$ are the same conjugacy maps from above.  The composition operator on this space is spectral.   Furthermore, by Poincar\'{e}'s linearization theorem \cite{Arnold:1988wx}, if the nonlinear dynamical system is a holomorphic diffeomorphism with a fixed point at zero and the linearization of the map at the fixed point has modulus strictly less than 1, then the conjugacy map $g\circ h$ is biholomorphic in a neighborhood of zero.  As long as for all $n\geq 0$, $f^{n}$ is invertible on its image and the inverse is analytic, the local conjugacy can be extended biholomorphically to the entire disc \cite{Lan:2012vw}.  Then the map $f \mapsto f \circ g \circ h$ defines a composition operator on $H^{2}(\D)$ (thm.\ 3.2.1(i), \cite{Singh:1993vt}).  Additionally, by part (iii) of the same theorem, this composition operator is a bijection.  Therefore, the pullback space of observables $\F$ for the analytic map is equivalent to $H^{2}(\D)$.
\end{rem}

\begin{rem}\label{rem:connection-with-work-of-cowen}
In Cowen and MacCluer \cite{Cowen:1994cw}, the spectra of composition operators on weighted Hardy spaces are investigated.  In one particular result, the authors prove that when $\Phi : \D \to \D$ is an univalent, holomorphic map that is not an automorphism and satisfies $\Phi(\D) \subset \D$ and $\Phi(0) = 0$, then the spectrum of the associated Koopman operator on (the classical Hardy space) $H^{2}(\D)$ is 
	\begin{equation}\label{eq:cowen-spectra}
	\sigma(U_{\Phi}) = \setc{ \lam }{ \abs{\lam} \leq \tilde{\rho}} \cup \setc{ \Phi'(0)^{k} }{ k=1,2,\dots} \cup \set{1},
	\end{equation}
where $\tilde{\rho}$ is the essential spectral radius of $\Phi$ and $\Phi'(0) \neq 0$ is the derivative of $\Phi$ at $0$.  Note that $\Phi'(0)$ is the eigenvalue of the linearized dynamics around the fixed point at 0.  The conditions put on $\Phi$ merely guarantee that 0 is a globally attracting (in the disc) fixed point.

Clearly, the last two sets of \eqref{eq:cowen-spectra} are equivalent to \eqref{eq:linearized-spectrum} and merely come from the linearized dynamics. The linear dynamics $w \mapsto \Phi'(0) w$ and the nonlinear dynamics $z \mapsto \Phi(z)$ are topologically conjugate in the disc under some diffeomorphism $g : \D \to \D$; i.e., $g\circ \Phi = \Phi'(0) \circ g$.  This paper defines the pullback space \eqref{eq:d1-pullback-space}, on which the Koopman operator corresponding to the nonlinear dynamics $\Phi$ is spectral, whereas Cowen and MacCluer consider the composition operator acting on $H^{2}(\D)$ and as a result obtain an additional term in the spectrum, namely $\setc{ \lam }{ \abs{\lam} \leq \tilde{\rho}}$.  Using the same argument as in the previous remark, if the conjugacy is biholomorphic in the disc, then $\F \equiv H^{2}(\D)$.  As we have shown, the composition operator has only point spectra.  This is a sharpening of the result by Cowen and MacCluer when applied to this specific class of maps.

Even if the conjugacy is not biholomorphic and we consider the composition operator $U_{\Phi}$ on $H^{2}(\D)$ rather than the pullback space $\F = \bhardy \circ g \circ h$, we can apply the GLA theorems in this paper to construct the eigenfunctions as long as $\tilde{\rho} < \Phi'(0)$.  In this case, $U_{\Phi}$ has a dominating point spectrum (take $R = \tilde{\rho}$ in definition \ref{defn:dominating-point-spectrum}).

\end{rem}

\subsubsection{Asymptotically stable limit cycles in $\mathbb{R}^{2}$.}
Consider a stable limit cycle in the plane, topologically conjugate in a neighborhood of the limit cycle to the linearized system
	\begin{equation}
	\begin{aligned}\label{eq:linear-limit-cycle}
	\dot{x} &= \rho(s) x \\
	\dot{s} &= 1 
	\end{aligned}
	\end{equation}
where $x \in\R$, $s \in S^{1} = \Z / 2\pi$, and $\rho(s) \in \R$ is $2\pi$ periodic.  Letting $\Phi_{t} : \R^{2} \to \R^{2}$ be the flow map of the system, the continuous time Koopman semigroup is formally defined as $U(t)f = f \circ \Phi_{t}$.  Eigenfunctions of the semigroup are functions $\phi$ taking the form $U(t) \phi = \lam^{t} \phi$.  We call $\lam \in \C$ an eigenvalue of the Koopman semigroup.

It was shown in \cite{Lan:2011wi} that the Koopman semigroup associated with the dynamical system \eqref{eq:linear-limit-cycle} has eigenfunctions of the form,
	\begin{equation}
	\begin{aligned}
	g_{m}(x,s) &= x^{m} e^{ -m  \int_{0}^{s} \left(\rho(\bar{s}) - \rho^{*} \right) d\bar{s}  } , \quad (m \in \N_{0})\\
	h_{n}(x,s) &= e^{i n s} , \quad (n \in \Z)
	\end{aligned}
	\end{equation}
with eigenvalues 
	\begin{align*}
	\lam_{m} &= e^{m\rho^{*}}, \text{ and} \\
	\mu_{n} &= e^{in},
	\end{align*}
respectively, and where $\rho^{*} = (2\pi)^{-1} \int_{0}^{2\pi} \rho(s) ds$.  We have that $\rho^{*} < 0$ since the limit cycle is asymptotically stable.

By the semigroup property of eigenfunctions, $b_{m,n}(x,s) = g_{m}(x,s) \cdot h_{n}(x,s)$ is an eigenfunction having eigenvalue $e^{(m\rho^{*}+in)}$.  Let $\mscr V$ be the subspace given by the linear span of elements of the form $b_{m,n} = g_{m} \cdot h_{n}$;
	\begin{equation}\label{eq:sum-efunc-limit-cycle}
	\begin{aligned}
	\mscr V &:= \setc{ \sum_{k=1}^{K} a_{k} (g_{m_{k}}\cdot h_{n_{k}})  }{ K\in \N, a_{k} \in \C, m_{k} \in \N_{0}, n_{k} \in \Z }  \\
	&= \setc{\sum_{k=0}^{K} a_{k}  x^{k} e^{ -k  \int_{0}^{s} \left(\rho(\bar{s})  - \rho^{*} \right)d\bar{s}  }  e^{i n_{k} s} }{ K\in \N_{0}, n_{k} \in \Z, a_{k} \in \C} .
	\end{aligned}
	\end{equation}

We show that $\mscr V$ is contained in a space of polynomials over a normed unital commutative ring $\mscr R$ that is dense in $L^{2}(S^{1},\mu)$, where $\mu$ is the normalized Haar measure $d\mu = (2\pi)^{-1} ds$.  Elements of $\mscr V$ can be written in the form 
	\begin{equation}\label{eq:linV}
	\mscr V = \setc{ \sum_{k=0}^{K}  x^{k}  e^{ -k  \int_{0}^{s} \left(\rho(\bar{s})  - \rho^{*} \right)d\bar{s}  }  \sum_{\abs{n} \leq N_{k}} a_{k,n}e^{ins} }{ K \in \N_{0}, a_{k,n} \in \C, n \in \Z, N_{k} \in \N_{0} }.
	\end{equation}
Define $\mscr W$ as
	\begin{equation}\label{eq:pre-ring-W}
	\mscr W := \setc{ e^{ -m \int_{0}^{s} \left( \rho(\bar{s})  - \rho^{*} \right) d\bar{s} } \sum_{\abs{n} \leq N} a_{n} e^{ins}}{ m \in \N_{0}, N \in \N_{0}, a_{n} \in \C}
	\end{equation}
and $\mscr R$ as
	\begin{equation}\label{eq:S1-normed-ring}
	\mscr R := \lin \mscr W =  \setc{ \sum_{k=1}^{K} b_{k} w_{k}(s)}{ K\in \N, w_{k} \in \mscr W, b_{k} \in \C} .
	\end{equation}

\begin{lem}
$\mscr R$ is a normed unital commutative ring under pointwise products of functions with norm $\norm{ f }_{2} = (\frac{1}{2\pi} \int_{0}^{2\pi} \abs{ f(s) }^{2} ds)^{1/2}$.
\end{lem}

\begin{proof}
The unit function $\mbf 1$, taking values $\mbf 1(s) = 1$, is contained in $\mscr W$ and has norm $\norm{ \mbf 1}_{2} = 1$; take $m = 0$ and $a_{n} = 0$ for $\abs{n}\neq 0$ and $a_{0} = 1$ in \eqref{eq:pre-ring-W}.  Similarly, for any $c \in \C$, $c\mbf 1$ is in $\mscr W$ and has finite norm.  This implies that the constant functions are in $\mscr R$.  Clearly, products of elements in $\mscr W$ commute.  This implies that products of elements from $\mscr R$ commute.  We must show that these products are also in $\mscr R$ and have finite norm.  Let $w_{j}$ and $w_{k}$ be in $\mscr W$.  It is a straightforward calculation to show that $(a_{j} w_{j}(s))( b_{k} w_{k}(s))$ is in $\mscr W$ for any $a_{j}, b_{k} \in \C$.  It follows that $\mscr R$ is closed under pointwise products.  Now let $g(s)$ be in $\mscr R$.  This function is a finite linear combinations of elements of $\mscr W$.  But elements of $\mscr W$ are continuous functions on the compact interval $[0, 2\pi]$.  Consequently, $g(s)$ is also a continuous function on the compact interval.  Therefore, there is some constant $C_{g}$ such that $\sup_{0 \leq s \leq 2\pi} \abs{g(s)} \leq C_{g}$.  Thus $\norm{ g }_{2} \leq C_{g}$.  Now, let $h(s)$ be another function in $\mscr R$.  We have shown that $g(s)\cdot h(s)$ is in $\mscr R$.  Furthermore, $g(s)h(s)$ is a continuous function on the compact interval $[0,2\pi]$.  Therefore, $\sup \abs{g(s)h(s)} \leq C_{gh}$ for $s \in [0,2\pi]$ and the product also has finite norm.
\end{proof}

\begin{lem}
$\mscr W$ is dense in $L^{2}(S^{1},\mu)$.  Consequently, $\mscr R$ is dense in $L^2(S^{1},\mu)$.
\end{lem}

\begin{proof}
$\set{ e^{i n s} }_{n \in \Z}$ is an orthonormal basis for $L^{2}(S^{1},\mu)$.  Put $\xi(s) = \int_{0}^{s}\left( \rho(\bar{s})  - \rho^{*}\right) d\bar{s}$, for $s \in [0,2\pi]$.  Since $\xi(s)$ is a continuous function on a closed interval, it is bounded; $\abs{\xi(s)} \leq M$ for all $s$.  Therefore, $0 < e^{-k M} \leq  e^{-k\xi(s)} \leq e^{k M}$, where $c_{k}$ is a constant depending on $k$.  It follows that $e^{k\xi(s)}$ is a positive bounded function bounded away from 0.

Fix $f \in L^{2}(S^{1},\mu)$.  Since $e^{k\xi(s)}$ is a bounded function, it follows that $g_{k}(s) := e^{k\xi(s)} f(s)$ is in $L^{2}(S^{1},\mu)$.  For each $\eps > 0$, there is a trigonometric polynomial such that $\norm{ g_{k} - \sum_{\abs{n}\leq N} a_{n} e^{ins} }\\< e^{ k M} \eps$.  Therefore, $\norm{ f - e^{-k \xi(s)} \sum_{\abs{n}\leq N} a_{n} e^{ins} } < \eps$.
\end{proof}

\begin{lem}
$\mscr V = \mscr R[x]$.
\end{lem}

\begin{proof}
This follows directly from the definitions of $\mscr V$, $\mscr W$, and $\mscr R$.
\end{proof}

Completing $\mscr R[x]$ under the polynomial norm \eqref{eq:polynomial-norm} gives the space consisting of elements of the form $\sum_{k=0}^{\infty} \psi_{k} x^{k}$ where $\psi_{k} \in L^{2}(S^{1},\mu)$, and $\sum_{k=0}^{\infty} \twonorm{ \psi_{k}}^{2} < \infty$.  These are well-defined functions for $\abs{x} < 1$.  The completion $\F := H_{L^{2}}^{2}$ is not a ring.  Furthermore, $\mscr V$ is dense in $H^{2}_{L^{2}}$.

\subsubsection{Spectrum and spectral measure for the Koopman semigroup on $H_{L^{2}}^{2}$.}
The linear space $\mscr V$ is generated by products of eigenfunctions of the Koopman semigroup and is dense in $H_{L^{2}}^{2}$.  Eigenvalues of these eigenfunctions are of the form 
	\begin{equation}\label{eq:product-evalues}
	\gamma_{k,n} = e^{(k \rho^{*} + i n)}, \quad (k\in \N_{0}, \forall n \in \Z) .
	\end{equation}
The set $\set{ \gamma_{k,n}}$ is dense on isolated rings in the complex plane.  Since $\rho^{*} < 1$, then $\abs{\gamma_{k,n}} \leq 1$.  It is easy to check that $\norm{ U(t) f }_{\mscr R,2} \leq \norm{f}_{\mscr R,2}$ for all $f \in \mscr V$ and $t\geq 0$.  Since $\mscr V$ is dense in $H_{L^{2}}^{2}$, we can extend $U(t)$ to $H_{L^{2}}^{2}$ by continuity.

Let $\mscr V_{K}$ be the subspace of $\mscr V$ consisting of all elements for which the degree of $x$ is at most $K$.  That is, in the definition \eqref{eq:linV} of $\mscr V$ fix $K$ rather than allowing it to run over all values of $\N_{0}$.

\begin{lem}
The Koopman semigroup leaves $\mscr V_{K}$ invariant.
\end{lem}

\begin{proof}
$\mscr V_{K}$ is generated by eigenfunctions of the Koopman operator.
\end{proof}

\begin{prop}
The Koopman semigroup corresponding to the dynamical system \eqref{eq:linear-limit-cycle} is a spectral operator of scalar type on $\F = H_{L^{2}}^{2}$.
\end{prop}

\begin{proof}
We show this by constructing explicitly the spectral measure $E$ so that $U(t) f = \int_{\C} z^{t} E(dz) f$ for $f \in H_{L^{2}}^{2}$.  Let $\mc B(\C)$ be the $\sigma$-algebra of Borel sets in the complex plane.  Let $D \in \mc B(\C)$ and define
	\begin{equation*}
	I(D) := \setc{ (k,n) \in \N_{0} \times \Z  }{ \gamma_{k,n} \in D},
	\end{equation*}
where $\gamma_{k,n}$ is defined in \eqref{eq:product-evalues}.  This is the set of all indices corresponding to the eigenvalues $\gamma_{k,n}$ that are contained in $D$.  Define the set function $E : \mc B(\C) \to \mscr L(\mscr V )$ as 
	\begin{equation}\label{eq:limit-cycle-spectral-measure}
	E(D)f = \sum_{ \substack{(k,n) \in I(D)\\ k\leq K, \abs{n} \leq N} } c_{k,n}(f) x^{k} e^{ -k \int_{0}^{s} \left( \rho(\bar{s}) - \rho^{*}\right)d\bar{s} }e^{i n s} ,
	\end{equation}
where
	\begin{equation*}
	f(x,s) = \sum_{k=0}^{K} \sum_{ \abs{n} \leq N } c_{k,n}(f) x^{k} e^{ -k \int_{0}^{s} \left( \rho(\bar{s}) - \rho^{*}\right)d\bar{s} }e^{i n s} 
	\end{equation*}
is the unique representation of $f \in \mscr V$.  The set function $E$ commutes with $U(t)$ for all $t$ and 
	\begin{equation*}
	\norm{ E(D) f }_{\mscr R, 2 }^{2} :=  \sum_{ \substack{(k,n) \in I(D)\\ k\leq K, \abs{n} \leq N} } \abs{c_{k,n}(f)}^{2} \twonorm{ e^{ -k \int_{0}^{s} \left( \rho(\bar{s}) - \rho^{*}\right)d\bar{s} }e^{i n s} }^{2} \leq \norm{ f }_{\mscr R, 2 }^{2}
	\end{equation*}
Therefore, $E$ is an equibounded set function commuting with $U(t)$.  It is easy to show that $E$ satisfies all the other properties of the spectral measure given in \eqref{defn:spectral-measure}.  Since $E(D)$ is a bounded linear operator on $\mscr V$ and $\mscr V$ is dense in $H^{2}_{L^{2}}$, it can be extended by continuity to all of $H_{L^{2}}^{2}$.
\end{proof}

\begin{rem}
The same pullback space of observables corresponding to bounded open sets in the basin of attraction can be constructed as was done for the attracting fixed point.  Theorem 2.6 of \cite{Lan:2012vw} guarantees the existence of the required topological conjugacy valid in the entire basin of attraction of the limit cycle for appropriately regular nonlinear maps $\Phi$ in $\R^{2}$ possessing a stable limit cycle.
\end{rem}


\section{Discussion and Conclusions}\label{sec:conclusions}
In the analysis of nonlinear dynamical systems, spectral analysis of the Koopman operator offers much insight into the system.  However, most of the spectral theory has dealt with the Koopman operator associated with a measure-preserving or non-dissipative system.  Many of the simplest examples of dynamical systems, and many practical systems, have dissipation or expansion present in them, and the traditional $L^{2}$-spectral theory of the measure-preserving case cannot handle these systems.  In order to fill this gap, this paper extended the theory to scalar-type spectral operators on locally convex linear topological spaces having a point spectrum that is not restricted to the unit circle.  Projections onto these eigenspaces can be recovered by Laplace averages of the operator.  When the point spectrum is contained in the unit circle, the Laplace averages of this paper reduce to the well-known Fourier averages used to compute projections for unitary operators in a Hilbert space.

The results in this paper, however, do not give a full spectral picture.  There are restrictions on the placement of both the eigenvalues and the continuous parts of the spectrum in the complex plane.  For example interleaving of the continuous and point spectrum cannot be immediately treated with the GLA theorems of this paper.   In particular, we must assume the point spectrum is concentrated on isolated circles in the complex plane and additionally the operator either has a dominating or minimal point spectrum.  Given some dynamical system, it is not clear a priori whether the corresponding Koopman operator satisfies these conditions and, if it does, what exactly a typical observable might look like.

These questions were answered here for finite-dimensional, dynamical systems possessing hyperbolic attracting fixed points in $\C^{d}$ and attracting limit cycles in $\R^{2}$.  When the dynamics are restricted to some bounded set in a basin of attraction for a fixed point, a space of observables can be constructed so that the spectrum of the Koopman operator possesses the properties assumed in this paper, namely that $\sigma$(U) has dominating or minimal point spectrum and eigenvalues are concentrated on isolated circles.  Formal eigenfunctions for the linearized system were used to generate a space of observables.  As an algebraic object, this space was a subspace of the space of polynomials over a normed unital commutative ring.  The indeterminates of these polynomials corresponded to coordinates corresponding to stable directions of the attractor while the coefficients took values in a normed ring consisting of observables supported on the attractor.  This polynomial space was completed using an $\ell^{2}$ polynomial norm to a Banach space that was identifiable with a space of observables defined on a bounded subset in the basin of attraction.  Observables for the topologically conjugate nonlinear system were merely the result of composing the space of observable for the Koopman operator of the linearized system with the conjugacy maps.  We conjecture that for any bounded attractor with a basin of attraction that can be properly ``coordinatized'', the natural space of observables for the Koopman operator will be a space of polynomials (and its completion) over a normed unital commutative ring, with the ring formed from observables supported on the attractor.

Numerical considerations were not treated in this paper.  Efficient numerical algorithms are needed to make the results of this paper useful in applied settings.  This will be the subject of future work.  To this end, we mention work that has been done on discrete Laplace transforms by Rokhlin \cite{Rokhlin:1988vk}, Strain \cite{Strain:1992jq}, and Anderson \cite{Andersson:2012ec}.  It may be that these algorithms can be leveraged into a package to compute the projections.  Speed and numerical stability will be primary concerns since these algorithms require good estimates of the eigenvalues.  Three other possible directions to pursue are the development of iterative methods using, for example, 
\begin{inparaenum}[(i)] \item Krylov subspace-type techniques and the methods of matrix-free linear algebra \cite{Trefethen:1997wg, Saad:2011tu}, 
\item using the conjugacy method to explicitly construct eigenfunctions and project an observable onto them, or
\item using polynomial approximations of indicator functions on the spectrum.  
\end{inparaenum}
The first technique is conceptually closer to Laplace averages in the sense that it would extract the most unstable modes first, whereas the second technique would allow the projection onto a particular eigenfunction at the cost of additional computational complexity in finding the eigenfunctions.  In general, the eigenfunctions are the roots of a nonlinear operator equation that can be solved using an operator version of the Newton method \cite{Katok:1997th}.  Furthermore, the discrete Laplace transforms and the matrix-free methods could be classified as a ``data-driven'' method, since, in principle, all that is required is a sequence of observations and not an explicit representation of the dynamical system.  The third technique rests on the operational calculus for spectral operators and that spectral projections are determined by integrating the spectral measure against indicator functions (see lemma \ref{eq:spectral-projections-lct}).  This final possibility, in principle, allows one to extract both the point spectrum and continuous parts of the spectrum even if the spectrum does not satisfy the conditions in this paper.  Furthermore, this is also a data driven technique as we only need iterations of an observable to compute this.  Questions revolve around best polynomial approximations to the indicator functions, which polynomial basis allows for stable and efficient computation of the approximating polynomial in addition to the parsimony of the approximation; i.e., for a fixed approximation error, which polynomial basis has an approximating polynomial satisfying the error bound with the polynomial having minimal degree?

In summary, this paper introduced a method for constructing eigenfunctions of a spectral operator that can be considered as an extension of mean ergodic theorems to dissipative or expanding systems; we showed certain basic examples of dynamical systems and spaces of observables on them that gave rise to spectral Koopman operators on which the GLA theorems could be applied; and, additionally, we showed that these spaces of observables were each identifiable with a completion of a space of polynomials over a normed unital commutative ring, with the ring elements being observables defined on the attractor and the ring being dense in separable, reflexive Banach space.  The observable spaces for the linearized dynamics took the form of a generalized Hardy space where instead of the coefficients being in $\C$, they took values in the Banach space.  The space of observables for the nonlinear dynamical system were constructed by composing the space of observables for the linearized system with the conjugacy map between the nonlinear and linear system.  The exact form of this pullback space is determined by the conjugacy map.

\bibliographystyle{amsplain}

\bibliography{gla}

\end{document}